\input ./style/arxiv-general.cfg
\documentclass[MSNbibl,number,citesort,seceqn,dvips]{arxbj}
\makeatletter
   \@ifpackageloaded{graphicx}{}{\usepackage{graphicx}}
\makeatother
\usepackage{mathbh}
\usepackage{upgreek}

\aid{0}
\volume{22}
\issue{1}
\pubyear{2016}
\firstpage{107}
\lastpage{142}
\doi{10.3150/14-BEJ624} % kopijuoti is 'New paper accepted'
\makeatletter
\newcommand{\mod}{\operatorname{mod}}
\newcommand{\E}{\mathbb E}
\newcommand{\R}{\mathbb{R}}
\newcommand{\Q}{\mathbb{Q}}
\newcommand{\bbR}{\mathbb{R}}
\newcommand{\T}{\mathbb{T}}
\newcommand{\N}{\mathbb{N}}
\newcommand{\bbN}{\mathbb{N}}
\newcommand{\Z}{\mathbb{Z}}
\newcommand{\bbZ}{\mathbb{Z}}
\renewcommand{\P}{\mathbb{P}}
\newcommand{\BB}{\mathcal{B}}
\newcommand{\LL}{\mathcal{L}}
\newcommand{\FF}{\mathcal{F}}
\newcommand{\ind}{\mathbh{1}}
\newtheorem{theorem}{Theorem}[section]
\newtheorem{lemma}[theorem]{Lemma}
\newtheorem{corollary}[theorem]{Corollary}
\newtheorem{proposition}[theorem]{Proposition}
\newproclaim{definition}[theorem]{Definition}
\newremark{remark}[theorem]{Remark}
\newremark{example}[theorem]{Example}
\makeatother

\begin{document}

\begin{frontmatter}

\title{Stochastic integral representations and classification of sum-
and max-infinitely divisible processes}
\runtitle{Stochastic integral representations of infinitely divisible
processes}

\begin{aug}
%%%% inicialai - be tarpu
\author[A]{\inits{Z.}\fnms{Zakhar}~\snm{Kabluchko}\thanksref{A}\ead[label=e1]{zakhar.kabluchko@uni-muenster.de}} \and
\author[B]{\inits{S.}\fnms{Stilian}~\snm{Stoev}\corref{}\thanksref{B}\ead[label=e2]{sstoev@umich.edu}}
%\author{\inits{}\fnms{}~\snm{}}
%%\runauthor{} %% auto
%\dedicated{}
\address[A]{Institut f\"ur Mathematische Statistik,
Universit\"at M\"unster,
Orl\'eans--Ring 10,
48149 M\"unster, Germany. \printead{e1}}
\address[B]{Department of Statistics,
University of Michigan, Ann Arbor,
439 W. Hall, 1085 S. University Avenue,
Ann Arbor, MI 48109-1107, USA.
\printead{e2}}
\end{aug}

% HISTORY:
\received{\smonth{7} \syear{2012}}
\revised{\smonth{2} \syear{2014}}

% ABSTRACT
%
\begin{abstract}
Introduced is the notion of minimality for spectral representations of
sum- and max-infinitely divisible processes and it is shown that the
minimal spectral
representation on a Borel space exists and is unique. This fact is used
to show that a stationary, stochastically continuous, sum- or
max-i.d. random
process on $\mathbb{R}^d$ can be generated by a measure-preserving
flow on a $\sigma$-finite Borel measure space and that this flow is
unique. This development
makes it possible to extend the classification program of Rosi\'nski
(\textit{Ann. Probab.} \textbf{23} (1995)  1163--1187) with a unified
treatment of both sum- and max-infinitely divisible processes.
As a particular case, a characterization of stationary, stochastically
continuous, union-infinitely divisible
random measurable subsets of $\mathbb{R}^d$ is obtained. Introduced and
classified are several new max-i.d. random field models including
fields of Penrose type
and fields associated to Poisson line processes.
\end{abstract}

% KEYWORDS
% visi is mazosios raides ir pagal abecele
%
\begin{keyword}
\kwd{infinitely divisible process}
\kwd{max-infinitely divisible process}
\kwd{measure-preserving flow}
\kwd{minimality}
\kwd{Poisson process}
\kwd{spectral representation}
\kwd{stochastic integral}
\end{keyword}
\end{frontmatter}

%s1 #&#
\section{Introduction} A stochastic process $X=\{X(t), t\in T\}$ is
called \textit{infinitely divisible} (i.d.) if for every $n\in\N$ it
can be represented (in distribution) as a sum of $n$ independent identically
distributed (i.i.d.) processes. On the other hand, $X$ is said to be
\textit{max-infinitely divisible} (max-i.d.) if for every $n\in\N$
it equals in distribution to the
pointwise maximum of $n$ i.i.d. processes. Infinitely divisible
distributions and random vectors have been extensively studied in the
literature, while the max-i.d. case is relatively less
known. The important and widely studied classes of stable and
max-stable processes arise as special cases of i.d. and max-i.d. processes, respectively. Recall that $X$ is stable if
for every $n\in\N$, the sum of $n$ independent copies of $X$ equals (in
distribution), a rescaled and shifted version of the original process.
The definition of
max-stable processes is similar, with the addition replaced by
componentwise maximum.

For the class of stable processes, a particularly rich representation
and classification theory based on the notion of stochastic integral
over a stable random measure was developed in the pioneering works of
Hardin~\cite{hardin} and Rosi\'nski~\cite{rosinski95} (see also~\cite
{rosinski2,kolodynskirosinski03,pipirastaqqu2004cy,pipirastaqqu2002st,samorodnitsky05}
and the book~\cite{samorodnitskytaqqu1994book}).
In parallel with the stable case, the works \cite
{dehaanpickands86,wangstoev10} developed analogous classification
theory for max-stable processes
based on stochastic max-integrals~\cite{dehaan84,stoevtaqqu2005}. In
fact, the close connection between the sum- and max-stable cases can
be formalized through the notion of association~\mbox{\cite{kabluchkoextremes,wangstoev09a}}.

In this paper, we develop a general structure and classification theory
that applies to both i.d. and max-i.d. processes in both discrete
and continuous time. It is based on stochastic integrals over
Poisson random measures. Our main motivation is to extend the
classification program for sum-stable processes pioneered by Rosi\'
nski~\cite{rosinski95} to the infinitely divisible
setting. We develop tools for the representation and study a variety of
i.d. and max-i.d. models from a unified perspective.

Stochastic integral representations of i.d. processes were developed
in the seminal works of Maruyama~\cite{maruyama70} and Rajput and
Rosi\'nski~\cite{rosinskirajput89}, while the max-i.d. case was
addressed by Balkema \textit{et~al.} \cite{balkemaetal93}. To the best
of our knowledge, the structure theory based on such stochastic integral
representations has not been much explored. A key problem in this
context is to determine how two spectral representations of the same
i.d. or max-i.d. process are related. In the special stable case,
this problem is related to the structure of the isometries of $L^\alpha
$-spaces~\cite{hardin0,hardin} and it was elegantly resolved in terms
of the notion of \emph{minimality}. In the setting of
i.d. processes, these methods are not available. Instead, we prove a
general result on the existence of conjugacy between equimeasurable
families of functions (Lemma~\ref{lft-point}, below)
and based on this result we define a corresponding notion of \emph{minimality}. It turns out that two minimal spectral representations of
the same i.d. process defined on $\sigma$-finite Borel spaces
are related through a unique measure space isomorphism between the two
spaces. This result is then used to show that a \emph{stationary} i.d. process can be generated as a stochastic integral over a
measure-preserving flow. This extends some classification results of
Rosi{\'n}ski~\cite{rosinski95,rosinski2} on the spectral
representations of stationary stable processes.
The i.d. theory we develop here is in fact simpler (although more
general) than the stable theory of~\cite{rosinski95,rosinski2} since we
deal with \textit{measure-preserving} rather than
\textit{non-singular} flows. Our results can be
specialized to the stable case by using the Maharam construction from
ergodic theory. This sheds more light on the subtle concept of
minimality in the stable case.

Recall that the law of a finite-dimensional i.d. random vector is
characterized by its L\'evy triplet~\cite{satobook}. Here, we focus on
the case where the i.d. random vectors have
trivial Gaussian component and their laws are determined by their L\'
evy measures along with constant location vectors. In a pioneering
work, Maruyama~\cite{maruyama70} considered the case of
i.d. processes $X= \{X(t), t\in T\}$ and extended the concept of L\'
evy measure to the infinitely dimensional setting as a measure on $\R
^T$. This extension is especially non-trivial when the set $T$ is not
countable because in this case several measurability issues arise. The
L\'evy measure introduced by Maruyama~\cite{maruyama70} could be used
to establish some of the
results of the present paper. Here, we chose to develop classification
theory based on spectral representations in order to draw parallels
with the abundant theory for stable processes.

Unlike the i.d. case, every one-dimensional distribution is
max-i.d., but the situation changes in higher dimensions.
Max-infinitely divisible random vectors are characterized by the so called
\emph{exponent measure}, which plays a role similar to that of the L\'
evy measure in the i.d. case (see, e.g., Chapter~5 in \cite
{resnickbook}). Two-dimensional max-i.d. distributions were
introduced by Balkema and Resnick \cite{balkemaresnick77}, the general
$d$-dimensional case was considered by Gerritse \cite{gerritse86} and
Vatan \cite{vatan85} (the latter work studies also max-i.d. vectors with
values in $\R^{\N}$). Representations in terms of suprema over a
Poisson point process were obtained by Gin\'e \textit{et~al.} \cite
{gineetal90} for max-i.d. processes with continuous sample paths and
by Balkema \textit{et~al.} \cite{balkemaetal93} for stochastically
continuous processes. It seems that in the case of uncountable $T$, the
general concept of exponent measure (parallel to~\cite{maruyama70}) has
not been studied. In this
paper, we develop the representation theory of max-i.d. processes
further, for example, by proving existence and uniqueness of the minimal
spectral representation and by constructing a representation over a
measure-preserving flow for stationary processes. We illustrate our
theory by introducing and classifying several new examples of
max-i.d. processes. As a special case, we arrive at a representation
result for stationary union infinitely divisible random sets,
which may be of independent interest. All our examples have direct
analogs in the sum-i.d. context.

\textit{The paper is organized as follows.} In Section~\ref{secmax-id},
we introduce minimal spectral representations for max-i.d. processes
and show their existence and uniqueness under
the general Condition~\textup{S} of separability in probability. Section~\ref
{secid} contains parallel results for i.d. processes. In Section~\ref
{secmeasver}, we discuss measurability of i.d. and max-i.d. processes.
In Section~\ref{secflow}, the developed theory
is used to associate stationary i.d. and max-i.d. processes to
measure preserving flows leading to extensions of known classification
results on stable processes.
In Section~\ref{secexamples}, we present several examples and applications.
The connection between the new notion of minimal spectral
representations and the existing ones for stable processes
is demonstrated in Section~\ref{secmax-stable}. In Sections~\ref{seciid}--\ref{secPenrose}, we present new examples of stationary
max-i.d. processes associated with
dissipative, conservative, or null flows. In Section~\ref{secU-id}, we
characterize the stationary union infinitely divisible random sets by
relating them to max-i.d.~processes.
The proofs are given in Section~\ref{secproofs}.

%s2 #&#
\section{Spectral representations of i.d. and max-i.d. processes}
%s2.1 #&#
\subsection{Spectral representations of max-i.d. processes} \label{secmax-id}
A stochastic process $X=\{X(t), t\in T\}$ defined on an index set $T$
and taking values in $\R$ is called max-i.d., if for all
$n\in\N$, there exist independent identically distributed (i.i.d.) processes $\{X_{i,n}(t), t\in T\}$, $i=1,\ldots,n$, such that%
%e2.1 #&#
%
\begin{equation}
\label{eqdefmaxidprocess} \bigl\{X(t), t\in T\bigr\} \stackrel{d} {=} \Bigl\{ \max
_{1\le i\le n} X_{i,n}(t), t\in T \Bigr\}.
\end{equation}
Here, $\stackrel{d}{=}$ denotes the equality of finite-dimensional
distributions. If $\{X(t), t\in T\}$ is a max-i.d. process, then for
every collection of non-decreasing functions $\varphi_t\dvtx \R\to\R$,
$t\in T$, the process $\{\varphi_t(X(t)), t\in T\}$ is also max-i.d.
By choosing
\[
\varphi_t(x) = \cases{ \mathrm{e}^{x}, &\quad if $\operatorname{essinf}  X(t) = -\infty $,
\cr
x-\operatorname{essinf}
X(t), &\quad if $\operatorname{essinf}  X(t) > -\infty$} %
\]
we can always achieve that $\operatorname{essinf} \varphi
_t(X(t))=0$.
In the sequel, we therefore assume without loss of generality that
$\operatorname{essinf}  X(t)=0$ for every $t\in T$.

Balkema \textit{et~al.} \cite{balkemaetal93} gave a representation of
max-i.d. processes in terms of stochastic max-integrals over a
Poisson point process. This representation is, in general, non-unique;
see Example~\ref{exnon-uniqueness} below. We will introduce the notion
of \textit{minimality} for representations of max-i.d. processes and
prove that the minimal representation exists and is unique.

We recall the construction of Balkema \textit{et~al.} \cite
{balkemaetal93} in a form which is suitable for our purposes. Let
$(\Omega,\BB,\mu)$ be a $\sigma$-finite measure space.
We denote by $\LL^{\vee}=\LL^{\vee}(\Omega,\BB,\mu)$ the space
of all
measurable functions $f\dvtx \Omega\to\R$ such that $f\geq0$ $\mu
$-a.e. and $\mu\{\omega\dvtx  f(\omega) > a\}$ is finite
for all $a>0$. As usual, two functions are identified if they differ on
a set of measure zero.
Note that
for every $f_1,f_2\in\LL^{\vee}$ and $c_1,c_2\geq0$ we have $\max
(c_1f_1,c_2f_2)\in\LL^{\vee}$.
Next let us recall the definition of the max-integral from~\cite
{balkemaetal93}. Let $\Pi_\mu=\{U_i, i\in J\}$ be a Poisson point
process on the space $(\Omega,\BB)$ with intensity $\mu$. Here, $J$ is
at most countable index set. For $f\in\LL^{\vee}$ define the
stochastic max-integral
%
%e2.2 #&#
%
\begin{equation}
\label{ee-int-def} I(f)\equiv\int^{\vee}_{\Omega} f\,\mathrm{d} \Pi_{\mu}:=\sup_{i\in J} f(U_i).
\end{equation}
Here, the supremum is taken over all atoms $U_i$ of the Poisson process
$\Pi_{\mu}$. If $\Pi_{\mu}$ is empty, which can happen if $\mu
(\Omega
)<\infty$, then the supremum in
the right-hand side is defined to be $0$. From~(\ref{ee-int-def}), one
readily derives a formula for the
joint distribution of the stochastic max-integrals: for all $f_1,\ldots,f_n\in\LL^\vee$ and $x_1,\ldots,x_n\geq0$ (not all of which are
$0$), we have
%
%e2.3 #&#
%
\begin{eqnarray}\label{eX-fdd}
\P\bigl\{ I(f_j) < x_j, 1\leq j\leq n\bigr\} &=& \P\Biggl\{ \Pi_{\mu}  \Biggl( \bigcup_{j=1}^n
\{ f_{j} \geq x_j\}  \Biggr) = 0  \Biggr\}
\nonumber\\[-8pt]\\[-8pt]
&=& \exp \Biggl\{ -\mu \Biggl(\bigcup_{j=1}^n
\{ f_{j} \geq x_j\} \Biggr)  \Biggr\} .\nonumber
\end{eqnarray}

Observe that for any collection of deterministic functions $f_t \in\LL
^\vee, t\in T$, the process $\{I(f_t), t\in T\}$,
is max-i.d. since the $X_{i,n}$'s in~(\ref{eqdefmaxidprocess}) can be
defined by using independent copies of the same stochastic
max-integrals but with respect to a Poisson point process with
intensity $\frac{1}n \mu$.

%
%de2.1 #&#
%
\begin{definition}
Let $X=\{X(t), t\in T\}$ be a max-i.d. process with $\operatorname{essinf}  X(t)=0$
for all $t\in T$. A~collection of functions $\{f_t, t\in T\}\subset\LL
^\vee(\Omega,\BB,\mu)$ is a \textit{spectral representation} of the
process $X$ if we have the following equality of laws:
%
%e2.4 #&#
%
\begin{equation}
\label{eqspecrepdef} \bigl\{X(t), t\in T\bigr\} \stackrel{d} {=} \biggl\{\int
^{\vee}_{\Omega} f_t \,\mathrm{d}\Pi_{\mu}, t
\in T \biggr\},
\end{equation}
where $\Pi_\mu$ is a Poisson point process on $(\Omega, \BB)$ with
intensity $\mu$.
\end{definition}

Here, we focus on the general class of processes that are separable in
probability in the sense of the following definition.

%
%de2.2 #&#
%
\begin{definition}
A stochastic process $\{X(t), t\in T\}$ satisfies Condition~\textup{S}
if there is an at most countable set $T_0\subset T$ such that for all
$t\in T$, there exists a sequence $\{t_n\}_{n\in\N}\subset T_0$, with
$X(t_n) \to X(t)$ in probability.
\end{definition}

As shown in Balkema \textit{et~al.} \cite{balkemaetal93}, the
convergence in probability for max-i.d. random variables is equivalent
to convergence in measure of
their spectral functions.

%
%pr2.3 #&#
%
\begin{proposition}[(Balkema \textit{et~al.} \cite{balkemaetal93})]
\label
{pmaxpconv} Let $f_n \in\LL^\vee(\Omega,\BB,\mu), n\in\N$. Then,
there is a random variable $\xi$ such that $I(f_n)\to\xi$ in
probability as $n\to\infty$, if and only if, there exists $f\in\LL
^\vee(\Omega, \BB,\mu)$ such that $f_n \to f$ in measure, as $n\to
\infty$. In this case, $\xi= I(f)$ a.s.
\end{proposition}

The proof follows from Theorems 4.4 and 4.5 in \cite{balkemaetal93}.

%
%re2.4 #&#
%
\begin{remark} \label{remmaxisometry} Observe that if $f_n\in\LL
^\vee
(\Omega,\BB,\mu)$ and $f_n\to f$, $n\to\infty$, in measure,
then necessarily $f\in\LL^\vee(\Omega,\BB,\mu)$. Indeed, since
for all
$\epsilon>0$, we have
$\{f>\epsilon\} \subset\{f_n>\epsilon/2\} \cup\{|f-f_n|>\epsilon
/2\}
$, it follows that $\mu\{f>\epsilon\}<\infty$, for all $\epsilon
>0$. Thus,
$\LL^\vee(\Omega,\BB,\mu)$ is closed with respect to convergence in
measure, which in fact can be metrized by a version of the Ky Fan metric:
%
%e2.5 #&#
%
\begin{equation}
\label{edmu} d_\mu(f,g):= \inf\bigl\{ \epsilon>0\dvtx  \mu\bigl(|f-g|\ge
\epsilon\bigr) \le \epsilon\bigr\}.
\end{equation}
Note that $\mu(|f-g|\ge\epsilon) <\infty$, for all $\epsilon>0$ and
$f, g\in\LL^\vee(\Omega,\BB,\mu)$. One can show that $\LL^\vee
(\Omega,\BB,\mu)$
equipped with $d_\mu$ becomes a complete metric space. Thus,
Proposition \ref{pmaxpconv} entails that the stochastic max-integral operator\vspace*{1pt}
is a homeomorphism of metric spaces. More precisely, $I\dvtx (\LL^\vee
(\Omega,\BB,\mu),d_\mu) \to(\LL^0(\P), d_{\mathrm{KF}})$ is a continuous bijection
onto its image with a continuous inverse. Here, $\LL^0(\P)$ is the
space of random variables on the probability space $(E,\FF,\P)$ on
which the Poisson process $\Pi_{\mu}$ is defined. The space $\LL
^0(\P)$
is endowed with the Ky Fan metric $d_{\mathrm{KF}}$ that metrizes the
convergence in probability (see, e.g., (\ref{eKyFan}) below).
\end{remark}

%
%th2.5 #&#
%
\begin{theorem}[(Balkema \textit{et~al.} \cite{balkemaetal93})]\label{theobalkema}
Let $\{X(t), t\in\R^d\}$ be a max-i.d. process satisfying Condition~\textup{S}.
There exists a spectral representation of $X$ defined on $\R$ endowed
with the Lebesgue measure.
\end{theorem}

We will prove existence and uniqueness of the spectral representation
under the following condition of minimality.

%de2.6 #&#
%
\begin{definition} \label{defmin}
A spectral representation $\{f_t, t\in T\}\subset\LL^{\vee}(\Omega,\BB,\mu)$ is called \textit{minimal} if the following two conditions hold:
\begin{enumerate}[(ii)]
\item[(i)] The $\sigma$-algebra generated by $\{f_t, t\in T\}$
coincides with $\BB$ up to $\mu$-zero sets. That is, for every $B\in
\BB$, exists an $A\in\sigma\{f_t, t\in T\}$, such that $\mu(A\Delta
B) =0$.
\item[(ii)] There is no set $B\in\BB$ such that $\mu(B)>0$ and for
every $t\in T$, $f_t=0$ a.e. on $B$.
\end{enumerate}
If just the second condition holds, we will say that the representation
has \textit{full support}.
\end{definition}

%
%re2.7 #&#
%
\begin{remark}
The first condition does not imply the second one: consider $\Omega=\{
0,1\}$ with counting measure, $T=\{1\}$, and $f_1(\omega)=\omega$.
\end{remark}

%
%th2.8 #&#
%
\begin{theorem}\label{theominspecrepmaxexists}
Let $X=\{X(t), t\in T\}$ be a max-i.d. process satisfying Condition~\textup{S}.
There exists a minimal spectral representation of $X$ defined on
$[0,1]$ endowed with a $\sigma$-finite Borel measure.
\end{theorem}

We recall next several notions of isomorphisms from measure theory. For
more details, see, for example, Chapter~22 in \cite{sikorskibook}, page~167 in \cite{halmosbook}, or Chapter~15.4 in \cite{roydenbook}.

%
%de2.9 #&#
%
\begin{definition} \label{defisomorphisms}
\textup{(i)} An \emph{isomorphism} between two \emph{measurable spaces}
$(\Omega_i,{\BB}_i), i=1,2$, is a bijection $\Phi\dvtx \Omega_1 \to
\Omega
_2$ such that both $\Phi$ and $\Phi^{-1}$ are measurable.

\textup{(ii)} A measurable space $(\Omega,{\BB})$ is said to be a \emph{Borel space} if it is isomorphic (in the sense of part~\textup{(i)}) to a complete
separable metric space endowed with its Borel $\sigma$-algebra.

\textup{(iii)}
A Borel space endowed with a $\sigma$-finite measure will be called a
\textit{$\sigma$-finite Borel space}.

\textup{(iv)} An \emph{isomorphism (modulo null sets)} between two \emph{measure spaces} $(\Omega_i,{\BB}_i,\mu_i)$, $i=1,2$, is a bijection
$\Phi\dvtx \Omega_1\setminus A_1\to\Omega_2\setminus A_2$, where
$A_1\in
\BB_1$ and $A_1\in\BB_2$ are null sets, such that both $\Phi$ and
$\Phi
^{-1}$ are measurable and $\mu_1(A) = \mu_2(\Phi(A))$, for all
measurable $A\subset\Omega_1\setminus A_1$.
Two isomorphisms $\Phi, \Psi$ are considered as equal modulo null sets
if $\Phi(\omega)=\Psi(\omega)$ for $\mu_1$-a.a. $\omega\in
\Omega_1$.
\end{definition}

%
%re2.10 #&#
%
\begin{remark} Any Borel space is isomorphic to the interval $[0,1]$ endowed
with the Borel \mbox{$\sigma$-}algebra or to an at most countable set endowed
with the $\sigma$-algebra of all subsets. This result is
known as Kuratowski's theorem (see, e.g., page~406 in \cite{roydenbook}).
\end{remark}

The next statement is the main result in this section. It establishes
the uniqueness of the minimal spectral representation.

%
%th2.11 #&#
%
\begin{theorem}\label{theominspecrepmaxunique}
Let $X=\{X(t), t\in T\}$ be a max-i.d. process. Let also $\{
f^{(i)}_t, t\in T\}$ be two minimal spectral representations of $X$
defined on the spaces $(\Omega_i,\mathcal B_i,\mu_i)$, $i=1,2$.
\begin{enumerate}[(ii)]
\item[(i)] If $(\Omega_1,\mathcal B_1,\mu_1)$ is a $\sigma$-finite
Borel space, then
there is a measurable map $\Phi\dvtx \Omega_2 \to\Omega_1$ such that
$\mu_1 = \mu_2\circ\Phi^{-1}$ and for all $t\in T$,
%
%e2.6 #&#
%
\begin{equation}
\label{eqtheominspecrep} f_t^{(2)}(\omega)=f_t^{(1)}
\circ \Phi(\omega)\qquad\mbox{for } \mu _2\mbox{-a.a. } \omega\in
\Omega_2.
\end{equation}

\item[(ii)] If both $(\Omega_i,\BB_i,\mu_i), i=1,2$ are $\sigma
$-finite Borel spaces, then
the mapping $\Phi$ in part \textup{(i)} is a measure space isomorphism and it
is unique modulo null sets.
\end{enumerate}
\end{theorem}

%
%ex2.12 #&#
%
\begin{example}\label{exnon-uniqueness}
Our definition of the space $\LL^\vee$ of integrands is more
restrictive than that of Balkema \textit{et~al.} \cite{balkemaetal93},
who allow measurable functions $f\dvtx \Omega\to\R$ with $\mu\{f > a\}
<\infty$, for \textit{some} $a\in\R$ (and do not assume that
$\operatorname{essinf}  X_t=0$).
With the definition used in~\cite{balkemaetal93} the uniqueness may
fail, even for minimal representations.
Indeed, let $\Omega_1=\Omega_2=\Z$ be endowed with the counting
measure. Set $T=\Z\cup\{*\}$ and define
\begin{eqnarray*}
f_t^{(1)}(\omega) &=& f_t^{(2)}(\omega)=
\ind_{\{t\}}(\omega)\qquad \mbox{if } t\neq*,
\\
f_*^{(1)}(\omega)&=&1,\qquad f_*^{(2)}(\omega) =\ind_{\{\omega>0\}}+\tfrac{1}2
\ind_{\{\omega\leq
0\}}. %
\end{eqnarray*}
One verifies readily that $\{f_t^{(i)}, t\in T\}$, $i=1,2$, are minimal
representations of the same \mbox{max-}i.d. process. However, there is no
bijection $\Phi\dvtx \Omega_1\to\Omega_2$ such that $f^{(1)}_*\circ\Phi
=f^{(2)}_*$. Note that $f_*^{(2)}\notin\LL^{\vee}$ and hence Theorem
\ref{theominspecrepmaxunique} does not apply. The constant $1/2$
in the definition of $f_*^{(2)}$ could in fact be replaced by any
$0<c<1$. This example shows why it is important to require that the
max-integrands $f$ in ${\mathcal L}^{\vee}$ satisfy
the condition $\mu\{f> a\}<\infty$ for \textit{all} $a>0$.
\end{example}

%s2.2 #&#
\subsection{Spectral representations of i.d. processes} \label{secid}

A process $\{X(t), t\in T\}$ is said to be infinitely divisible (i.d. or sum-i.d.) if for all $n\in\N$ it can be represented (in
distribution) as a sum of $n$ independent and identically distributed processes.

There is already a lot of literature on the spectral representations of
i.d. processes (see, e.g.,~\cite{maruyama70,roy07,rosinskirajput89}).
Our aim here is to study the minimality and the uniqueness of the
spectral representation. This is a key step which allows us to
extend the classification program pioneered by Rosi\'nski \cite
{rosinski95} in the stable case to the general i.d. context.

Let $(\Omega,\BB,\mu)$ be a $\sigma$-finite measure space. The
space of
integrands $\LL^+$ consists of all measurable
$f\dvtx \Omega\to\R$ such that
%
%e2.7 #&#
%
\begin{equation}
\label{eL+def} \int_{\Omega} \min\bigl\{\varepsilon, \bigl|f(
\omega)\bigr|^2\bigr\}\mu(\mathrm{d}\omega )<+\infty
\end{equation}
for some (or, equivalently, any) $\varepsilon>0$. Functions differing
on a set
of measure $0$ are identified. Observe that $\LL^+$ is a linear space
since $1\wedge(f+g)^2 \le2 (1\wedge f^2 + 1\wedge g^2)$.
Following Maruyama~\cite{maruyama70}, for $f\in\LL^+$ define the
stochastic integral
%
%e2.8 #&#
%
\begin{equation}
\label{eqdefidstochint} I(f)\equiv\int_\Omega^+ f\,\mathrm{d} \Pi_\mu:= \lim_{\varepsilon\to0+}  \biggl\{ \sum_{i\in J}
f(U_i)\ind_{\{|f(U_i)|>\varepsilon\}} - \int_{\{
|f|>\varepsilon\}} a(f)\,\mathrm{d} \mu
 \biggr\},
\end{equation}
where $\Pi_\mu=\{U_i, i\in J\}$ is a Poisson point process on
$(\Omega,\BB)$ with intensity $\mu$ and
%
%e2.9 #&#
%
\begin{equation}
\label{ea-def} a(u) =\cases{ u, &\quad $|u|\le1$,
\cr
1, &\quad $u>1$,
\cr
-1, &
\quad $u< -1$.}
\end{equation}
Note that the limit in~(\ref{eqdefidstochint}) exists in the a.s. sense by the convergence theorem for $L^2$-bounded martingales. For
$f_1,\ldots,f_n\in\LL^+$ the joint distribution of the $I(f_j)$'s is
characterized as follows. For all $\theta_1,\ldots,\theta_n\in\R$
we have
%
%e2.10 #&#
%
\begin{equation}
\label{ejoint-chf} \E \mathrm{e}^{\mathrm{i}\sum_{j=1}^n \theta_j I(f_j)} = \exp \Biggl\{ \int_{\Omega}
 \Biggl( \mathrm{e}^{\mathrm{i}\sum_{j=1}^n \theta_j f_j(\omega)} - 1 - \mathrm{i} \sum_{j=1}^n
\theta_j a\bigl(f_j(\omega)\bigr) \Biggr) \mu(\mathrm{d} \omega)  \Biggr\},
\end{equation}
where i stands for the imaginary unit. In particular, it is easy to verify that for all $f,g\in\LL^+$ and
$c\in
\R$ we have $I(f+g)=I(f)+I(g)+{\rm const}$ and
$I(cf)=cI(f) +{\rm const}$, that is, the functional $I$ is essentially
linear up to additive constants (see also
(\ref{egamma-def}) below).

%
%de2.13 #&#
%
\begin{definition}
Let $X=\{X(t), t\in T\}$ be an i.d. process with trivial Gaussian
component defined on some index set $T$.

\begin{enumerate}[(ii)]
\item[(i)] A
collection of functions $\{f_t, t\in T\}\subset\LL^+(\Omega, \BB,
\mu
)$ is a spectral representation of the process $X$ if
we have the following equality of laws:
%
%e2.11 #&#
%
\begin{equation}
\label{eqspecrepdefsum} \bigl\{X(t), t\in T\bigr\} \stackrel{d} {=} \biggl\{\int
^{+}_{\Omega} f_t \,\mathrm{d}\Pi
_{\mu}+c(t), t\in T \biggr\},
\end{equation}
where $\Pi_{\mu}$ is a Poisson point process on $(\Omega, \BB,\mu
)$ and
$c\dvtx T\to\R$ is some function.

\item[(ii)] The spectral representation is called \emph{minimal} if $\{
f_t, t\in T\}$ satisfy both conditions of Definition~\ref{defmin}.
\end{enumerate}
\end{definition}

%
%th2.14 #&#
%
\begin{theorem}\label{theominspecrepsumexists}
Let $\{X(t), t\in T\}$ be an i.d. process which has a trivial Gaussian
component and satisfies Condition~\textup{S}.
There exists a minimal spectral representation of $X$ defined on
$[0,1]$ endowed with a $\sigma$-finite Borel measure.
\end{theorem}

The proof of the above result (given in Section~\ref{secproofs} below) utilizes
the following truncated \mbox{$L^2$-}metric on the space $\LL^+$:
%
%e2.12 #&#
%
\begin{equation}
\label{ed-def} d(f,g) \equiv d(f-g):=  \biggl( \int_{\Omega} 1
\wedge(f-g)^2 \,\mathrm{d}\mu  \biggr)^{1/2},\qquad f,g\in\LL^+.
\end{equation}
Note that the triangle inequality follows from $1\wedge|f+g| \le
1\wedge|f| + 1\wedge|g|$ and the triangle inequality
in $L^2(\Omega)$.

%pr2.15 #&#
%
\begin{proposition}\label{pLLpluscomplete}
For any $\sigma$-finite Borel space $(\Omega,{\BB},\mu)$, the space
$(\LL^+,d)$ is separable and complete.
\end{proposition}

The next proposition shows that the metric $d$ on the space $\LL^+$
corresponds to convergence in probability on the space of stochastic
integrals $\{I(f), f\in\LL^+\}$.

%pr2.16 #&#
%
\begin{proposition}\label{pcauchyprobabcauchyLplus}
For $f_n \in\LL^+$ and $c_n\in\R$, we have that $I(f_n) + c_n$
converges to a random variable $\xi$ in probability, as $n\to\infty$,
if and only if, there exists some
$f\in\LL^+$ and $c\in\R$, such that $d(f_n-f) + |c_n-c| \to0$, as
$n\to\infty$. In this case, $\xi= I(f) + c$ a.s.
\end{proposition}

The next theorem, which is analogous to Theorem~\ref
{theominspecrepmaxunique}, shows the uniqueness of the minimal spectral
representation for i.d. processes.

%
%th2.17 #&#
%
\begin{theorem}\label{theominspecrepsumunique}
Let $X=\{X(t), t\in T\}$ be an i.d. process. Let also $\{f^{(i)}_t,
t\in T\}$ be two minimal spectral representations of $X$
defined on the spaces $(\Omega_i,\mathcal B_i,\mu_i)$, $i=1,2$.
\begin{enumerate}[(ii)]
\item[(i)] If $(\Omega_1,\mathcal B_1,\mu_1)$ is a $\sigma$-finite
Borel space, then
there is a measurable map $\Phi\dvtx \Omega_2 \to\Omega_1$ such that
$\mu_1 = \mu_2\circ\Phi^{-1}$ and for all $t\in T$,
%
%e2.13 #&#
%
\begin{equation}
\label{eqtheominspecrepsum} f_t^{(2)}(\omega)=f_t^{(1)}
\circ \Phi(\omega)\qquad\mbox{for } \mu _2\mbox{-a.a. } \omega\in
\Omega_2.
\end{equation}

\item[(ii)] If both $(\Omega_i,\BB_i,\mu_i), i=1,2$ are $\sigma
$-finite Borel spaces, then
the mapping $\Phi$ in part \textup{(i)} is a measure space isomorphism and it
is unique modulo null sets.
\end{enumerate}
\end{theorem}

%
%re2.18 #&#
%
\begin{remark} Theorems \ref{theominspecrepmaxunique} and \ref
{theominspecrepsumunique} require that
both representations be minimal. Minimality can be enforced by
replacing $\Omega_i$ by $\operatorname{supp}\{f_t^{(i)}, t\in T\}$, $i=1,2$, and
letting $\BB_i = \sigma\{f_t^{(i)}, t\in T\}$, $i=1,2$. To be able to
apply the above results, however, at least one of the measure spaces
should be Borel. This is neither automatic nor obvious for the new
spaces $(\Omega_i,\BB_i)$. If both spaces are Borel, then by part~\textup{(ii)}
of Theorems \ref{theominspecrepmaxunique} and~\ref
{theominspecrepsumunique} minimality ensures that the two representations
are related through a unique measure space isomorphism as in (\ref
{eqtheominspecrepsum}). In fact, by part~\textup{(i)} of Theorems \ref
{theominspecrepmaxunique} and~\ref{theominspecrepsumunique}, this
relation still holds
for some not necessarily invertible map, provided that just the first
representation is minimal. This last fact
is an important technical tool, analogous to Remark 2.5 in Rosi\'nski
\cite{rosinski95}, in the stable case.
\end{remark}

%s2.3 #&#
\subsection{Measurability and stochastic continuity}\label{secmeasver}

When studying path properties or ergodicity, it is important or in fact
necessary to work with measurable processes. Here, we establish
necessary and sufficient
conditions for the existence of measurable versions of max-i.d. and
i.d. processes in terms of their spectral representations.

Let $(T,\rho_T)$ be a separable metric space, equipped with its Borel
$\sigma$-algebra ${\mathcal A}$. Consider a family of measurable
functions $\{f_t, t\in T\}$ on $(\Omega,\BB,\mu)$. This family is said
to be \textit{jointly measurable} if the map $(t,\omega)\mapsto
f_t(\omega)$
is measurable with respect to the product $\sigma$-algebra ${\mathcal
A}\otimes\BB:= \sigma({\mathcal A}\times\BB)$.
For the classical notions of measurability and strong separability of a
stochastic process $X=\{X_t\}_{t\in T}$, we refer to Chapter~9 in \cite
{samorodnitskytaqqu1994book}.
The following result extends Proposition 4.1 in \cite{wangstoev10} (see
also Theorem 11.1.1 in \cite{samorodnitskytaqqu1994book}).
Its proof is given in \cite{KabluchkoStoev2014-arxiv}.

%
%pr2.19 #&#
%
\begin{proposition}\label{pmeasurability} Let $X=\{X(t), t\in T\}$ be
a max-i.d. (i.d., resp.) process with spectral representation
$\{f_t, t\in T\}\subset\LL^{\vee/+}(\Omega,\BB,\mu)$ over a
$\sigma
$-finite measure space $(\Omega,\BB,\mu)$ as in (\ref{eqspecrepdef})
(as in~(\ref{eqspecrepdefsum}), resp.). The process $X$ has a
measurable modification if and only if the following two conditions hold:
\begin{enumerate}[(ii)]
\item[(i)] The family $\{f_t, t\in T\}$ has a jointly measurable
modification, that is, there exists a \mbox{${\mathcal A}\otimes\BB$-}measurable mapping
$(t,\omega
)\mapsto g_t(\omega)$, such that $f_t = g_t$ ($\mod \mu$), for all
$t\in T$.
\item[(ii)] The function $t\mapsto c(t)$ is measurable if $X$ is i.d.
\end{enumerate}
If the process $X$ has a measurable modification, then it satisfies
Condition~\textup{S}, and consequently, $X$ has a spectral representation
over a $\sigma$-finite Borel space. In this case, the measurable
version of $X$ and the jointly measurable version of $\{f_t, t\in T\}$
may be taken to be strongly separable.
\end{proposition}

%
%re2.20 #&#
%
\begin{remark} The last result shows that, if $X$ has a measurable
version, then this version as well as its corresponding jointly measurable
representation $\{f_t\}_{t\in T}$ can be taken to be strongly separable
(cf. Chapter~9 in \cite{samorodnitskytaqqu1994book}).
\end{remark}

%
%re2.21 #&#
%
\begin{remark} \label{remmeasstat} Proposition 3.1 in~\cite{roy2010}
states that any \emph{measurable} and \emph{stationary} random field $\{
X(t), t\in\R^d\}$ is
automatically continuous in probability.

This follows from a celebrated result due to Banach on Polish groups.
Conversely, it is well known that stochastically continuous processes
(indexed by separable metric spaces) have measurable modifications.
Therefore, in the case
of stationary random fields on $\R^d$ the assumptions of stochastic
continuity and measurability are essentially equivalent.
\end{remark}

%s3 #&#
\section{Flow representations and ergodic decompositions for stationary i.d. and max-i.d. processes} \label{secflow}

In the stable and max-stable cases, the connections between spectral
representations and ergodic theoretic decompositions of
the underlying flows have lead to a wealth of decomposition and
classification results. We will show that this theory naturally
extends to the i.d. and max-i.d. setting. An alternative powerful
approach from the perspective of \emph{Poisson suspensions} and \emph{factor maps} has
been recently pioneered by Emmanuel Roy~\cite{roy07,roymaharam}. We
expect that these tools can be used to develop an all-encompassing
theory, but
this is beyond the scope and goals of the present work.
Here, we adopt an alternative approach, which is useful when the i.d. and max-i.d. processes are given through their stochastic integral
representations.

%s3.1 #&#
\subsection{Existence and uniqueness of flow representations}
Let $\T$ denote either $\Z$ or $\R$. Consider the measure space $(\T
^d,\mathcal A,\lambda)$, where $\lambda$ is either the counting
measure if
$\T=\Z$ or the Lebesgue measure with
${\mathcal A}$ the Borel $\sigma$-algebra if $\T=\R$. A
measure-preserving $\T^d$-action (or flow) on a\vspace*{1pt} measure space $(\Omega, \BB, \mu)$ is a family $\{T_t\}_{t\in\T^d}$ of measure space
isomorphisms $T_t\dvtx \Omega\to\Omega$ such that $T_0=\mbox{id}$ $\mu
$-a.e. and for every $t,s\in\R^d$, $T_t\circ T_s = T_{t+s}$ $\mu
$-a.e. The action is called \emph{measurable} if
$(t, \omega)\mapsto T_t(\omega)$ is a measurable map from $\T
^d\times
\Omega$ to~$\Omega$, where the former space is endowed with the product
$\sigma$-algebra
${\mathcal A}\otimes\BB$.

The next statement combines both the i.d. and max-i.d. cases and
shows that one can associate stationary processes with
measure-preserving actions. The common theme is the
uniqueness.

%
%th3.1 #&#
%
\begin{theorem}\label{theoflowrepexists}
Let $X=\{X(t), t\in\T^d\}$ be a stationary and stochastically
continuous max-i.d. (resp., i.d., without Gaussian component)
process with a representation
$\{f_t, t\in T\} \subset\LL^{\vee/+}(\Omega,\BB,\mu)$ as in
(\ref{eqspecrepdef}) (or (\ref{eqspecrepdefsum}), resp.).
If the representation is minimal and the measure space $(\Omega,\BB,\mu
)$ is $\sigma$-finite Borel, then there exists a measurable and
measure-preserving flow
$\{T_t\}_{t\in\T^d}$ on $(\Omega,\BB,\mu)$ such that for all $t\in
\T
^d$, we have
%
%e3.1 #&#
%
\begin{equation}
\label{eqspecrepstat} f_t = f_0 \circ T_t,\qquad\mu \mbox{-a.e.}
\end{equation}
In the i.d. case the function $c(t)$ in~(\ref{eqspecrepdefsum}) is constant.
\end{theorem}

\begin{pf} By stationarity, for every fixed $s\in\T^d$, both $\{f_t, t\in\T^d\}$ and
$\{f_{t+s}, t\in\T^d \}$ are minimal spectral representations of $X$
defined over the same $\sigma$-finite Borel space. By
Theorems~\ref{theominspecrepmaxunique}\textup{(ii)} and \ref{theominspecrepsumunique}\textup{(ii)}, there is a modulo $\mu$ unique
automorphism $T_s$ of the measure space $(\Omega, \mathcal B, \mu)$
such that for every $t\in\T^d$, $f_{s+t}=f_t\circ T_s$, $\mu$-a.e.
Let us show that for every $s_1,s_2\in\T^d$,
$T_{s_1+s_2}=T_{s_1}\circ
T_{s_2}$, $\mu$-a.e. Indeed, we have for every $t\in\T^d$,
\[
f_t\circ T_{s_1+s_2}=f_{s_1+s_2+t}=f_t
\circ(T_{s_1}\circ T_{s_2}),\qquad\mu\mbox{-a.e.}
\]
By the uniqueness of the automorphisms, we have
$T_{s_1+s_2}=T_{s_1}\circ T_{s_2}$, $\mu$-a.e., which yields~(\ref
{eqspecrepstat}).
In the sum-i.d. case, note also that the term $c(t)$ appearing
in~(\ref{eqspecrepdefsum}) does not depend on $t\in\T^d$ by stationarity.

This completes the proof in the case $\T=\Z$. In the case $\T=\R$
the flow $\{T_t\}_{t\in\R^d}$ constructed in this way need not in
general be measurable
(see, e.g., Example \ref{examplenon-measuarable}, below). However, one
can argue as
in~\cite{rosinski95}, by using the works of Mackey \cite{mackey62} and
Sikorski~\cite{sikorskibook}, that each $T_t$
can be modified on a set of $\mu$-measure zero so that the flow
property is valid with probability one and the flow
becomes measurable. Indeed, observe first that by Proposition \ref
{pmeasurability}, we may assume that the
representation $\{f_t, t\in\R^d\}$ is jointly measurable. Now,
consider the Boolean $\sigma$-algebra ${\mathcal B}_\mu$
whose elements are equivalence classes $[B]$ of sets $B\in{\mathcal B}$
with respect to the equality modulo $\mu$-null sets. Following
the argument on page 1168 of Rosi\'nski~\cite{rosinski95}, in order to
apply Theorem 1 of \cite{mackey62}, it is enough
to show that for every finite measure $\widetilde\nu$ on the Boolean
$\sigma$-algebra ${\mathcal B}_\mu$ and for every set $B\in
{\mathcal B}$ the function
\[
t\mapsto\widetilde\nu\bigl( \bigl[T_t(B)\bigr]\bigr) %
\]
is Borel measurable. For the finite measure $\nu$ on $(\Omega,{\mathcal B})$,
induced by $\widetilde\nu$ as $\nu(B):= \widetilde\nu([B])$, we have
$\nu\ll\mu$ and hence
%
%e3.2 #&#
%
\begin{equation}
\label{enu-as-indicator} \widetilde\nu\bigl(\bigl[T_t(B)\bigr]\bigr) = \int
_{\Omega} (\ind_{B}\circ T_{-t}) (\omega ) k(
\omega) \mu(\mathrm{d}\omega),
\end{equation}
where $k = \mathrm{d}\nu/\mathrm{d}\mu\in L^1(\Omega, {\mathcal B},\mu)$ is the
Radon--Nikodym density.

By minimality of the representation $\{f_t, t\in\R^d\}$, we have that
there is a set $A\in{\mathcal B}$ such that $\mu(A\Delta B)=0$ and
$A \in
\sigma\{f_{t_i}, i\in\bbN\}$, for some countable collection $\{
t_i, i\in\N\} \subset\R^d$ and therefore,
there exists a Borel function $g\dvtx \R^{\N}\to\R$, such that $\ind_A =
g(f_{t_1},f_{t_2},\ldots)$. We will show that
the integral in (\ref{enu-as-indicator}) is a Borel measurable function
of $t$. Since $f_{t_i-t} = f_{t_i}\circ T_{-t}
\mod\mu$, we have that $\ind_A \circ T_{-t} =
g(f_{t_1-t},f_{t_2-t},\ldots)\mod\mu$, for every $t\in\R^d$.
Hence, we have
%
%e3.3 #&#
%
\begin{equation}
\label{eB-via-g} \widetilde\nu\bigl(\bigl[T_t(B)\bigr]\bigr) = \int
_{\Omega} g\bigl(f_{t_1-t}(\omega ),f_{t_2-t}(
\omega),\ldots\bigr) k(\omega) \mu(\mathrm{d}\omega)\qquad\mbox {for all } t\in
\R^d.
\end{equation}
Now, the joint measurability of $\{f_t, t\in\R^d\}$ implies the joint
measurability of the integrand on the right-hand side of (\ref
{eB-via-g}) as a function of $t$ and $\omega$.
Hence, by Fubini's theorem, applied to (\ref{eB-via-g}), we obtain that
$t\mapsto\widetilde\nu([T_t(B)])$ is Borel measurable. Now,
proceeding as on page~1169 of~\cite{rosinski95}, by using
Theorem~1 of~\cite{mackey62} and Theorem~32.5 of~\cite{sikorskibook},
we obtain that the flow $\{T_t, t\in\R^d\}$ has
a jointly measurable modification.
\end{pf}

Theorem \ref{theoflowrepexists}, as in the stable case (cf.~\cite
{rosinski95}) motivates the following.

%
%de3.2 #&#
%
\begin{definition}\label{defgen-by-flow} A stationary max-i.d. or
i.d. process (with trivial Gaussian part) $X$ with spectral
representation $\{f_t, t\in\T^d\} \subset\LL^{\vee/+}(\Omega,\BB,\mu
)$ as in~(\ref{eqspecrepdef}) or
(\ref{eqspecrepdefsum}), is said to be generated by a
measure-preserving flow $\{T_t\}_{t\in\T^d}$ on $(\Omega,\BB,\mu)$ if:
\begin{enumerate}[(ii)]
\item[(i)] For every $t\in\T^d$, $f_t = f_0 \circ T_t$ $\mu$-a.e.
\item[(ii)] $\{f_t, t\in\T^d\}$ has full support.
\end{enumerate}
In this case, we call the pair $(f_0, \{T_t\}_{t\in\R^d})$ a flow
representation of $X$ on $(\Omega, \BB,\mu)$.
Furthermore, if $\{f_t, t\in\T^d\}$ is minimal, then we say that the
flow representation is minimal.
\end{definition}

%
%re3.3 #&#
%
\begin{remark}
According to the above definition, a flow representation need not be
minimal. The reason why we consider more general non-minimal flow
representations is because
minimality may not be easy to check or ensure in applications.
\end{remark}

The next corollary follows immediately from Theorems~\ref
{theominspecrepmaxexists}, \ref{theominspecrepsumexists},
and \ref{theoflowrepexists}.

%
%co3.4 #&#
%
\begin{corollary}
Let $\{X(t), t\in\T^d\}$ be a stationary max-i.d. or i.d. (without
Gaussian component) process which is stochastically continuous
(equivalently: has a
measurable modification). Then, $X$ has a minimal representation by a
measurable flow on a $\sigma$-finite Borel space.
\end{corollary}

The next result shows that the minimal flow representation associated
with a stationary stochastically continuous max-i.d. or i.d. process
is essentially unique
up to a \emph{flow isomorphism}. This fact will allow us to obtain
structural results about the above two types of processes from ergodic
theoretic properties of the
associated flows.

%
%th3.5 #&#
%
\begin{theorem}\label{theoflowrepunique}
Let $X=\{X(t), t\in\T^d\}$ be a stationary max-i.d. or i.d. (without
Gaussian component) random field. If $\T=\R$, suppose\vspace*{1pt} in addition that
$X$ is stochastically continuous.
If $(f_0^{(i)},\{T_t^{(i)}\}_{t\in\T^d})$ are two minimal flow
representations of $X$ on $\sigma$-finite Borel
spaces $(\Omega_i,\BB_i,\mu_i)$, $i=1,2$, then there is a measure space
isomorphism $\Phi\dvtx \Omega_1\to\Omega_2$ (defined modulo null sets) such
that $f_0^{(1)}=f_0^{(2)}\circ\Phi$,
$\mu_1$-a.e., and for all $t\in\T^d$,
%
%e3.4 #&#
%
\begin{equation}
\label{eqtheoflowrepunique} \Phi\circ T_t^{(1)} = T_t^{(2)}
\circ\Phi,\qquad \mu_1\mbox{-a.e.}
\end{equation}
The isomorphism $\Phi$ is unique modulo null sets.
\end{theorem}

\begin{pf}
By assumption, $\{f_0^{(i)} \circ T_t^{(i)}, t\in\T^d\}$, $i=1,2$, are
two minimal spectral representations of $X$. By the uniqueness of the
minimal spectral representations over
Borel spaces, there is a (modulo null sets) unique measure space
isomorphism $\Phi\dvtx \Omega_1\to\Omega_2$ such that for every $t\in\T^d$,
%
%e3.5 #&#
%
\begin{equation}
\label{eqflowunique1} \bigl(f_0^{(1)}\circ T_t^{(1)}
\bigr) = \bigl(f_0^{(2)}\circ T_t^{(2)}
\bigr) \circ\Phi,\qquad \mu_1\mbox{-a.e.}
\end{equation}
Replacing $t$ by $t+s$ and taking the composition of both sides with
$T_{-s}^{(1)}$ from the right, we obtain
%
%e3.6 #&#
%
\begin{equation}
\label{eqflowunique2} \bigl(f_0^{(1)}\circ T_t^{(1)}
\bigr) = \bigl(f_0^{(2)}\circ T_t^{(2)}
\bigr) \circ \bigl(T_s^{(2)} \circ\Phi\circ
T_{-s}^{(1)}\bigr),\qquad \mu_1\mbox{-a.e.}
\end{equation}
It follows from~(\ref{eqflowunique1}) and~(\ref{eqflowunique2}) that
$\Phi$ and
$T_s^{(2)}\circ\Phi\circ T_{-s}^{(1)}$ are two measure space
isomorphisms each linking the representations
$\{f_0^{(i)} \circ T_t^{(i)}, t\in\T^d\}$, $i=1,2$. By the last
statement of Theorem~\ref{theominspecrepmaxunique}
or Theorem~\ref{theominspecrepsumunique}, these isomorphisms should be
equal up to $\mu_1$-zero sets.
This yields~(\ref{eqtheoflowrepunique}).
\end{pf}

The following example shows that the stochastic continuity of the
process $X$ is an essential assumption for the
measurability of the flow in Theorem~\ref{theoflowrepexists}.

%
%ex3.6 #&#
%
\begin{example} \label{examplenon-measuarable}
Take $\Omega=\R$, let $\BB$ be the Borel $\sigma$-algebra and $\mu
=\lambda$ the Lebesgue measure.
Take any function $f_0\in\LL^{\vee}(\R, \mathcal B, \lambda)$, for
concreteness let $f_0(\omega)=\mathrm{e}^{-\omega}\ind_{\omega>0}$. We now
construct a measure-preserving flow on $(\R, {\mathcal B}, \lambda)$ with
``bad'' properties. Let $\varphi\dvtx \R\to\R$ be a Hamel
function, that is a \textit{non-measurable} function which satisfies
the Cauchy functional equation
$\varphi(t+s)=\varphi(t)+\varphi(s)$ for all $t,s\in\R$. Define a map
$T_t\dvtx \R\to\R$ by $T_t(\omega)=\omega-\varphi(t)$, for $s,t\in\R
$. It is
easy to check that $\{T_t\}_{t\in\R}$ is a measure-preserving (but not
measurable) flow on $(\R, \mathcal B, \lambda)$. Consider now a
max-i.d. process $\{X(t), t\in\R\}$ defined by $X(t)=I(f_t)$, where
\[
f_t(\omega) = (f_0\circ T_t) (\omega) =
\mathrm{e}^{-(\omega-\varphi(t))}\ind _{\omega>\varphi(t)}. %
\]

The process $X$ defined above is a stationary max-i.d. process which
satisfies Condition~\textup{S}. To see that Condition~\textup{S} is satisfied note that
$\varphi(t)=ct$ for all $t\in\Q$ and some constant $c=\varphi(1)$.
Hence, the collection $\{X(t), t\in\Q\}$ is dense in probability in
$\{
X(t), t\in\R\}$. The spectral representation $\{f_t, t\in\R\}$
constructed above is minimal and it is defined on a $\sigma$-finite
Borel space. Note that the minimality follows from the fact that the
$\sigma$-algebra generated by the functions $\mathrm{e}^{-(x-ct)}\ind_{x>ct}$,
$t\in\Q$, coincides with $\mathcal B$. This representation is
generated by a measure-preserving flow $\{T_t\}_{t\in\R}$. On the
other hand, the process $X$ is not stochastically continuous
(otherwise, the function $\varphi$ would be continuous and hence,
linear). By Remark~\ref{remmeasstat}, the process $X$ has no
jointly measurable modification. Consequently, by Proposition~\ref
{pmeasurability}, the representation
$\{f_t, t\in\R\}$ has no jointly measurable modification and the
process $X$ has no representation generated by a measurable
measure-preserving flow.
\end{example}

%s3.2 #&#
\subsection{Conservative--dissipative decompositions} \label{secHopf}

Let $\{T_t\}_{t\in\T^d}$, with $\T=\Z$ or $\R$, be a
measure-preserving and measurable $\T^d$-action on
a $\sigma$-finite Borel space $(\Omega,\BB,\mu)$. Suppose first that
$\T= \Z$ and $d=1$. Recall that a set $W\in\BB$ is said to be
wandering if $T_n(W)$, $n\in\Z$, are
disjoint modulo $\mu$. If $\Omega= \bigcup_{t\in\Z} T_n(W) \mod\mu
$, for
some (maximal) wandering set $W$, then the flow is said to be dissipative.
Conversely, a flow $\{T_n, n\in\Z\}$ is said to be conservative if it
has no wandering sets of positive measure.
In general, a flow may be neither purely conservative nor dissipative.
The Hopf decomposition entails that $\Omega= C \cup D$,
where $C\cap D =\varnothing$ and $C$ and $D$ are two $T_1$-invariant sets
such that the restriction of $T_1$ is conservative on $C$ and dissipative
on $D$. In the continuous-parameter case $\T= \R$ and $d=1$, one can
show that the Hopf decompositions $\Omega= C_t\cup D_t$ corresponding
to the measure
preserving map $T_{t}$ do not depend on $t\in\R\setminus\{0\}$, modulo
$\mu$ (cf.~\cite{rosinski95,krengel68,krengel1985}). Furthermore, a
celebrated result due to Krengel implies
that $\{T_t, t\in\R\}$ is dissipative if and only if it is isomorphic
to a mixture of Lebesgue shifts, that is, $T_t\circ\Phi(s,v) = \Phi
(s+t,v)$ for a measure
space isomorphism $\Phi\dvtx  (\T\times V, \mathcal A\otimes{\mathcal
V},\lambda\,\mathrm{d}\nu) \to(\Omega,\BB,\mu)$.

The multi-parameter case $d\ge2$ is more delicate since it is not
obvious how to even define the conservative/dissipative component of
the flow.
In a series of works Rosi\'nski, Samorodnitsky and Parthanil Roy \cite
{rosinski2,roysamorodnitsky2008,roy2010}
have shown that the Hopf decomposition and Krengel's characterization
of dissipativity extend to the multi-parameter setting with $\T$
discrete and/or continuous. Here, we shall adopt the approach of
Parthanil Roy \cite{roy2010} and say that the conservative
(dissipative) component of the flow
$\{T_t, t\in\R^d\}$ is that of the discrete skeleton $\{T_\gamma, \gamma\in\Z^d\}$ (see Proposition~2.1 therein). The flow is said to be
conservative (dissipative, resp.) if its dissipative (conservative, resp.) component is trivial. The
following characterization result may be taken as a definition of the
Hopf decomposition of a
$\T^d$-action.

%
%th3.7 #&#
%
\begin{theorem}[(Corollary 2.2 in \cite{roy2010})] \label{teoRoy-C-test}
Let $\{T_t, t\in\T^d\}$ be a measure-preserving and measurable flow.
Let also $h\in L^1(\Omega,\BB,\mu)$ be positive $\mu$-a.e. Then the
conservative part of $\{T_t\}_{t\in\T^d}$ is modulo $\mu$ equal to:
\[
C:=  \biggl\{ \omega\in\Omega\dvtx  \int_{\T^d} h
\bigl(T_t(\omega)\bigr) \lambda (\mathrm{d}t) = \infty \biggr\}. %
\]
And for the dissipative component, we have $D = \Omega\setminus C$.
\end{theorem}

%
%de3.8 #&#
%
\begin{definition}
Let $X = \{X_t, t\in\T^d\}$ be a measurable stationary max-i.d. or
i.d. random field (without Gaussian part) generated by a measurable
flow $T=\{T_t, t\in\T^d\}$ on $(\Omega,\BB,\mu)$ in the sense of
Definition \ref{defgen-by-flow}. We shall say that $X$ is generated by
a conservative
(dissipative) flow if $\{T_t, t\in\T^d\}$ is conservative (dissipative).
\end{definition}

The following result shows that this definition does not depend on the
choice of the flow representation. It provides, moreover,
a useful integral test for identifying the conservative and dissipative
parts of a flow. The situation is conceptually similar to the stable
and max-stable
cases (Corollary 4.2 in \cite{rosinski95}, Proposition 3.2 in \cite
{roy2010}, or Theorem 5.2 in \cite{wangstoev10}).

%
%th3.9 #&#
%
\begin{theorem}\label{teoC-int-test} Consider a stationary
max-i.d.~(or i.d.) process $X$ with a jointly measurable spectral
representation $\{f_t, t\in\T^d\}$ of full support.
The process $X$ is generated by a conservative flow, if and only if,
for every (equivalently any) nonnegative Borel function $\psi\dvtx [0,\infty
)\to[0,\infty)$ such
that $\psi(x)>0$ for all $x>0$, and $\int_\Omega\psi(|f_0|) \,\mathrm{d}\mu
<\infty$, we have
%
%e3.7 #&#
%
\begin{equation}
\label{eC-test} \int_{\T^d} \psi\bigl( \bigl|f_t(
\omega)\bigr| \bigr) \lambda(\mathrm{d}t) = \infty\qquad \mbox{for } \mu\mbox{-a.e. } \omega.
\end{equation}
Conversely, the process $X$ is generated by a dissipative flow, if the
latter integral is finite $\mu$-a.e. for some (equivalently, every)
$\psi$ such that
$\psi(x)>0$ for all $x>0$ and $\int_\Omega\psi(|f_0|) \,\mathrm{d}\mu<\infty$.
\end{theorem}

Consider a max-i.d. or an i.d. process $X= \{X(t), t\in\T^d\}$
with a measurable spectral representation of full support.
Motivated by Theorem \ref{teoC-int-test}, let
%
%e3.8 #&#
%
\begin{equation}
\label{eC-D-def} C =  \biggl\{ \omega\in\Omega\dvtx  \int_{\T^d}
\psi\bigl(\bigl|f_t(\omega)\bigr|\bigr) \lambda(\mathrm{d}t) = \infty \biggr\}\quad\mbox{and}\quad D:= \Omega \setminus C.
\end{equation}
By restricting the spectral representation to the sets $C$ and $D$, we
obtain the following decomposition of
$X$ into a max/sum of two independent processes:
%
%e3.9 #&#
%
\begin{equation}
\label{eC-D} X \stackrel{d} {=} X_C \,\Box\, X_D,\qquad \Box\in
\{\vee,+\},
\end{equation}
where $X_C(t):= I^{\vee/+}(\ind_C f_t) + c$ and $X_D(t):= I^{\vee
/+}(\ind_D f_t) + c$, $t\in\T^d$, where $c=0$ in the max-i.d. case.

By relation (\ref{ef-via-g}), as in the proof of Theorem \ref
{teoC-int-test}, it follows that $X_C$ and $X_D$ are \emph{stationary} and
generated by conservative and dissipative flows, respectively. The next
result shows that this decomposition is unique.

%
%co3.10 #&#
%
\begin{corollary} The conservative/dissipative decomposition (\ref
{eC-D}) is unique in law.
\end{corollary}

\begin{pf} Let $\{f_t^{(i)}, t\in\T^d\}$ be two measurable
representations of $X$ of full support defined on $(\Omega_i,\BB
_i,\mu
_i), i=1,2$.
As in the proof of Theorem \ref{teoC-int-test}, there exist
measure-preserving $\Phi_i\dvtx \Omega_i\to\widetilde \Omega, i=1,2$, such
that for all
$t\in\T^d$, we have $f_t^{(i)}(\omega)=g_t(\Phi_i(\omega))$,
modulo $\mu
_i$ $(= \widetilde \mu\circ\Phi_i^{-1}), i=1,2$, where $g_t=
g_0\circ
T_t, t\in T$
is a measurable\vspace*{1pt} minimal spectral representation over the Borel $\sigma
$-finite space $(\widetilde \Omega,\widetilde \BB,\widetilde \mu
)$. Let $C_i$
and $D_i$
be defined\vspace*{1pt} as in (\ref{eC-D-def}) with $f$ replaced by $f^{(i)}, i=1,2$. Then, as argued in the proof of Theorem \ref{teoC-int-test} (cf.
(\ref{ef-via-g})),
we have that $C_i = \Phi_i^{-1}(\widetilde C)$, modulo~$\mu_i$, where
$\widetilde C$ is the conservative part of the flow $\{T_t, t\in\T
^d\}$.
This fact since the
$\Phi_i$'s are  measure preserving implies that $\{\ind_{\widetilde C}
g_t, t\in\T^d\}$ is a spectral representation for both $X_{C_1}$ and $X_{C_2}$,
where $X_{C_i}(t) = I^{\vee/+}(\ind_{C_i}f_t^{(i)}) + c, i=1,2$.
Hence, $X_{C_1} \stackrel{d}{=} X_{C_2}$. One can similarly show that
\mbox{$X_{D_1}\stackrel{d}{=}X_{D_2}$}.
\end{pf}

In view of (the multi-parameter version) of Krengel's characterization
of dissipativity (see, e.g., \cite{rosinski2} or Corollary 2.4 in
\cite
{roy2010}),
we arrive at the following important result.

%
%co3.11 #&#
%
\begin{corollary}\label{cKrengel} A measurable stationary max-i.d. or
i.d. process $X$ is generated by a dissipative flow if and only if
it has a mixed moving maximum/average representation. That is, for some
\mbox{$\sigma$-}finite Borel space $(V,{\mathcal V},\nu)$ and
$(\Omega,\BB,\mu) = (\T^d \times V, {\mathcal A} \otimes
{\mathcal V}, \lambda
\otimes\nu)$, we have
\[
X \stackrel{d} {=} \bigl\{ I^{\vee/+}(f_t) + c, t\in
\T^d \bigr\},\qquad \mbox{where } f_t(s,v) =
f_0(t+s,v), (s,v)\in\Omega
\]
for some $f_0\in\LL^{\vee/+}(\Omega,\BB,\mu)$.
\end{corollary}

%s4 #&#
\section{Examples of max-i.d. processes}\label{secexamples}

%s4.1 #&#
\subsection{Max-stable processes}\label{secmax-stable}
Max-stable processes form a subclass of the max-i.d. processes. Fix
$\alpha>0$. A process $X=\{X(t), t\in T\}$ is called ($\alpha$-Fr\'
echet) max-stable if for every $n\in\N$ the process $X_1\vee\cdots
\vee
X_n$ has the same law as $n^{1/\alpha} X$, where $X_1,\ldots,X_n$ are
i.i.d. copies of $X$. The marginal distributions of $X$ are $\alpha
$-Fr\'echet distributions of the form $\P[X(t)\leq x] = \exp\{-\sigma
^{\alpha}(t) x^{-\alpha}\}$, $x>0$. Here, $\sigma(t)>0$ is called the
scale parameter of $X(t)$.

For max-stable processes, a theory of spectral representations has
been developed; see~\cite
{dehaan84,dehaanpickands86,stoev07,kabluchkoextremes,wangstoev10}. We
will explain the connection to the max-i.d. spectral representations
developed here.
Let $L^{\alpha}_+(\Omega',\BB',\mu')$ be the set of measurable
functions $g\dvtx \Omega'\to[0,\infty)$ such that $\int_{\Omega'}
g^{\alpha}\,\mathrm{d}\mu'<\infty$. Let
$\Omega=(0,\infty) \times\Omega'$ be equipped with the product
$\sigma
$-algebra $\BB$ and with a~measure $\mu=\alpha u^{-(\alpha+1)} \,\mathrm{d}u
\mu
'$. A collection of functions $\{g_t, t\in T\}\subset L^{\alpha
}_+(\Omega',\BB',\mu')$ is called a spectral representation of a
max-stable process $\{X(t), t\in T\}$ if
%
%e4.1 #&#
%
\begin{equation}
\label{eqspecrepmaxstabledef} \bigl\{X(t), t\in T\bigr\} \stackrel{d} {=} \biggl\{\bigvee
_{i\in\bbN} u_i g_t\bigl(
\omega_i'\bigr), t\in T \biggr\},
\end{equation}
where $\{(u_i, \omega_i'), i\in\N\}$, are points of the Poisson process
$\Pi_{\mu}$ with intensity $\mu$ on $\Omega$.
The process~$X$, being also max-i.d., must admit a spectral
representation in the sense of Section~\ref{secmax-id}. This
representation can be constructed as follows. Define $f_t\dvtx \Omega\to
[0,\infty)$ by $f_t(u, \omega')= u g_t(\omega')$, $t\in T$. Then,
(\ref{eqspecrepmaxstabledef}) implies that $\{f_t, t\in T\}$ is a spectral
representation of $X$ viewed as a max-i.d. process.

A spectral representation $\{g_t, t\in T\}\subset L^{\alpha}_+(\Omega
',\BB',\mu')$ of a max-stable process $X$ is called minimal
(see~\cite
{kabluchkoextremes,wangstoev09a,wangstoev10}) if \textup{(i)} $\operatorname{supp} \{ g_t, t\in T\} = \Omega' \mod\mu'$ and
\textup{(ii)} $\sigma\{g_t/g_s, t,s\in T\} = \BB' \mod\mu'$.

%le4.1 #&#
%
\begin{lemma}\label{lminimality}
In the above context, if $\{g_t, t\in T\}$ is a minimal spectral
representation of a max-stable process $X$, then $\{f_t, t\in T\}$ is
a minimal spectral representation of $X$ as a max-i.d. process.
\end{lemma}

\begin{pf}
Notice that $(f_t/f_s)(x,y) = (g_t/g_s)(y)$, (with $0/0$ is interpreted
as $0$) does not
depend on $x$. Therefore, $\rho(F):= \sigma\{f_t/f_s, t,s\in T\} =
\bbR
_+ \times\sigma\{ g_t/g_s, t,s \in T\}$, which is ($\mod \mu$)
equivalent to
$\bbR_+\times \BB_{\Omega'}:= \{\bbR_+ \times B\dvtx   B\in\BB
_{\Omega'}\}$ by condition~\textup{(ii)}.
We also have that $g_t$ is ($\mod \mu$) measurable with respect to
(w.r.t.) $\bbR_+\times \BB_{\Omega'}$ ($\mod \mu$)
and since $\rho(F) = \bbR_+\times \BB_{\Omega'}$ ($\mod \mu$), it
follows that $g_t$ is ($\mod \mu$) measurable
w.r.t.~$\sigma(F):= \sigma\{ f_t, t\in T\}$ $(\supset\rho(F))$. Therefore,
$(x,y)\mapsto x\ind_{\{\operatorname{supp}(g_t)\}} (y) = f_t(x,y)/g_t(y)$ is
($\mod \mu$) measurable
w.r.t.~$\sigma(F)$. Now, the full support condition \textup{(i)} implies also
that $(x,y)\mapsto x$ is ($\mod \mu$) measurable w.r.t.~$\sigma(F)$.
This implies that
$\mathop{\BB}_{\bbR_+}\times\,\Omega'$ is included in $\sigma(F)$ ($\mod
\mu$). Since also $\bbR_+\times \BB_{\Omega'}$ is ($\mod \mu$)
included in $\sigma(F)$,
it follows that ${\BB}_{\bbR_+}\times \BB_{\Omega'}$ is ($\mod \mu$)
contained in $\sigma(F)$. This shows that
${\BB}_{\bbR_+}\otimes{\BB}_{\Omega'} \equiv\sigma({\BB}_{\bbR
_+}\times \BB_{\Omega'}) = \sigma(F)$ ($\mod \mu$). This shows that
$\{f_t,t\in T\}$ is a minimal representation of $X$ since condition~\textup{(i)}
for $\{g_t, t\in T\}$ implies
also the full support condition for the~$f_t$'s.
\end{pf}

%
%re4.2 #&#
%
\begin{remark}
Lemma \ref{lminimality} shows that all previous results on max-stable
processes that rely on the notion of minimality
can be obtained via the new notion of minimality. The following
construction due to Maharam gives the precise
connection between the ``old'' and ``new'' spectral representations in
the case of stationary processes.
\end{remark}

Let $X = \{X(t), t\in\R^d\}$ be a stationary stochastically continuous
max-stable process
with $\alpha$-Fr\'echet margins. Then by~\cite{dehaanpickands86},
there is a \emph{non-singular} flow $T'_t$ on a $\sigma$-finite Borel
space $(\Omega',{\BB}',\mu')$
and a function $g_0 \in L_+^\alpha(\Omega',{\BB}',\mu')$ such that
$\{
g_t, t\in\R^d\}$ is a minimal spectral representation of $X$, where
\[
g_t =  \biggl( \frac{\mathrm{d} \mu' \circ T_t'}{\mathrm{d} \mu'}  \biggr)^{1/\alpha}
g_0 \circ T_t',\qquad t\in\R. %
\]
(Recall that a measurable flow $\{T'_t\}_{t\in\R^d}$ is said to be
non-singular if the measures $\mu'\circ T'_t$ and $\mu$ are equivalent.)

The process $X$, being max-stable, is also max-i.d. Let us construct
the flow representation of $X$ in the sense of Section~\ref{secflow}.
We shall employ the
Maharam construction~\cite{maharam64,aaronsonbook}. Let $\Omega=
(0,\infty) \times\Omega'$ and consider the mappings $T_t\dvtx \Omega\to
\Omega$ defined by
%
%e4.2 #&#
%
\begin{equation}
\label{eMaharam} T_t \bigl(u,\omega'\bigr):= \biggl(
\biggl( \frac{\mathrm{d}\mu'\circ T_t'}{\mathrm{d}\mu
'}\bigl(\omega'\bigr) \biggr)^{1/\alpha}
u, T'_t \bigl(\omega'\bigr) \biggr),
\end{equation}
where $(t,\omega')\mapsto{\mathrm{d}(\mu'\circ T_t')}/{\mathrm{d}\mu'}(\omega')$ is a
measurable version of the Radon--Nikodym
derivatives (see, e.g., Theorem A.1 in~\cite{kolodynskirosinski03}). It
is easy to see that $\{T_t\}_{t\in\mathbb R}$ is a measurable flow,
which is measure-preserving (see, e.g., \cite
{maharam64,aaronsonbook}). Now, (\ref{eMaharam}) implies that $(f_0, \{
T_t\}_{t\in\R^d})$ is a flow representation of $X$ in the sense of
Section~\ref{secflow}.

%s4.2 #&#
\subsection{Independent random variables} \label{seciid}
Any collection $\{X(t), t\in T\}$ of independent random variables forms
a max-i.d. process. To see this, take any $n\in\N$ and let
$X_{i,n}(t), 1\leq i\leq n, t\in T$, be independent random variables
such that $\P[X_{i,n}(t)\leq x]=(\P[X(t)\leq x])^{1/n}$. Then, $\{X(t),
t\in T\}$ has the same law as $\{\bigvee_{i=1}^n X_{i,n}(t), t\in T\}$,
thus showing the max-i.d. property. Assume that $T$ is countable.
Then, Condition~\textup{S} is satisfied. The minimal spectral representation
of $\{X(t), t\in T\}$ can be constructed as follows. As always, we
assume that $\operatorname{essinf}  X(t)=0$ and,
additionally, $\P[X(t)=0]<1$ for all
$t\in T$. Let $\Omega=T\times(0,\infty)$ be endowed with the product
of the power set $2^T$ and the Borel $\sigma$-algebra on $(0,\infty)$.
Define a measure $\mu$ on $\Omega$ by $\mu(\{t\}\times[x,\infty
))=-\log\P[X(t) < x]$, $t\in T$, $x>0$. In this way, $\Omega$ turns
into a $\sigma$-finite Borel space.
Define the functions $f_t\dvtx \Omega\to\R$, $t\in T$, by
\[
f_t(s,x)= %
\cases{ x, &\quad$t=s$,
\cr
0, &\quad$t\neq
s$,}\hspace*{16pt}
 s\in T, x>0. %
\]
Then, $\{f_t, t\in T\}$ is a minimal spectral representation of $\{X_t,
t\in T\}$. If $T=\Z$ and the random variables $X_t$ are i.i.d., then
$X$ is stationary and we can define a (discrete time) flow
representation by setting $T_t(s,x)=(s-t,x)$, $t\in\Z$, and noting that
$f_t=f_0\circ T_t$.

%s4.3 #&#
\subsection{Mixed moving maximum processes} \label{secmmm}

Here, we present a general probabilistic construction of mixed moving
maxima max-i.d. processes. Similar construction
applies in the i.d. context. Let $\{U_i, i\in\N\}$ (interpreted as
storm centers) be the points of a Poisson process on $\R^d$ with constant
intensity $\lambda$.
Let $\{F(t), t\in\R^d\}$ be a measurable random process with values in
$[0,\infty)$ such that for every $a>0$, we have
%
%e4.3 #&#
%
\begin{equation}
\label{eMM-integrability-condition} \int_{\R^d} \P\bigl[F(t)>a\bigr] \,\mathrm{d}t <\infty.
\end{equation}
Let $F_n, n\in\N$, be i.i.d. copies of $F$ (storms), which are
independent from the Poisson process $\{U_i, i\in\N\}$ of storm
centers. Define a process $\{X(t), t\in\R^d\}$ by
%
%e4.4 #&#
%
\begin{equation}
\label{emmm-def} X(t)=\sup_{i\in\N} F_i(t-U_i).
\end{equation}
Condition (\ref{eMM-integrability-condition}) implies that $X$ is a
well-defined max-i.d. process, which is stationary by the translation
invariance
of the point process $\{U_i, i\in\N\}$. Indeed, without loss of
generality, we can let $F_i(t) = f(t,V_i)$, where $V_i$ are i.i.d. $\operatorname{Uniform}(0,1)$ random variables
and $f\dvtx \R^d\times[0,1] \to[0,\infty)$ is a Borel function. Thus,
$\Pi
_\mu= \{(U_i,V_i), i\in\N\}$ is a Poisson point process on
$\bbR^d\times[0,1]$ with intensity $\mu(\mathrm{d}u\,\mathrm{d}v) = \lambda \,\mathrm{d}u \,\mathrm{d}v$.
Relation (\ref{eMM-integrability-condition}) and Fubini's theorem guarantee
that $f_t(u,v):= f(t-u,v) \in\LL^+(\bbR^d\times[0,1],\mu)$, for all
$t\in\bbR^d$, and hence
\[
\bigl\{X(t), t\in\R^d\bigr\} \stackrel{d} {=}  \biggl\{\int
_{\bbR^d\times
[0,1]}^\vee f(t-u,v)\Pi_\mu(\mathrm{d}u,\mathrm{d}v) \biggr\} %
\]
is well-defined. Clearly $f_t(u,v) = f_0(T_t(u,v))$, where $f_0(u,v):=
f(-u,v)$ and $T_{t}(u,v):= (u-t,v), t\in\bbR^d$, is the simple
Lebesgue shift
flow in the first coordinate, which is measurable and
measure-preserving. This shows that the process $X$ is stationary and
in fact has a \textit{mixed moving
maximum} representation. The above discussion and Corollary \ref
{cKrengel} imply that the process $X$ in (\ref{emmm-def}) is generated
by a dissipative flow.

%s4.4 #&#
\subsection{Max-i.d. processes associated to Poisson line processes}\label{secPoisson-line}
Instead of taking points of a Poisson process as storm centers in (\ref
{emmm-def}), we can also take lines of a Poisson line process as storm centers.
Let $\mathbb T = \R/ (2\uppi\Z)$ be identified with the unit circle.
Each point $(r,\varphi)$ in $\mathbb L:= \R\times\mathbb T$
corresponds to an oriented line in $\R^2$
which passes through the point $(r\cos\varphi, r \sin\varphi)$ in the
direction of the vector $(-\sin\varphi, \cos\varphi)$. In this way,
$\mathbb L$ can be identified with the set
of all oriented lines in $\R^2$. Take a Poisson point process $\{
(r_i,\varphi_i), i\in\N\}$ on $\mathbb L$ whose
intensity is $\lambda \,\mathrm{d}r \times \mathrm{d}\varphi$, where $\lambda>0$ is constant.
The corresponding random set of lines in $\R^2$ is called the Poisson
line process and is interpreted as the set of storm centers. Its law is
invariant with
respect to translations of $\R^2$; see, for example, \cite{kingmanbook}.

Let now $\{F(t), t\in\R\}$ be a measurable process with values in
$[0,\infty)$ such that for every $a>0$,
we have $\int_\R\P[F(t)>a] \,\mathrm{d}t <\infty$. Let $F_i$, $i\in\N$, be
i.i.d. copies of $F$ (storms). Define a process $\{\eta(x,y),
(x,y)\in
\R^2\}$ by
%
%e4.5 #&#
%
\begin{equation}
\label{eetaPoissonline} \eta(x,y)= \max_{i\in\N} F_i (x\cos
\varphi_i + y\sin\varphi_i - r_i).
\end{equation}
Note that $|x\cos\varphi+y\sin\varphi-r|$ is the distance from the point
$(x,y)$ to the line corresponding to $(r,\varphi)\in\mathbb L$.
As in the previous Section~\ref{secmmm}, without loss of generality, we
have $F_i(t)=f(t,V_i)$, where $V_i$ are i.i.d. $\operatorname{Uniform}(0,1)$ random
variables and hence
%
%e4.6 #&#
%
\begin{eqnarray}\label{eeta-x,y-f-rep}
&& \bigl\{\eta(x,y), (x,y)\in\R^2\bigr\}
\nonumber\\[-8pt]\\[-8pt]
&&\quad \stackrel{d} {=} \biggl\{ \int_{\R
\times[0,2\uppi)\times[0,1]}^\vee f(x\cos\varphi+ y\sin
\varphi- r,v) \Pi_\mu(\mathrm{d}r,\mathrm{d}\varphi, \mathrm{d}v), (x,y)\in\R^2 \biggr\},\nonumber
\end{eqnarray}
where $\Pi_\mu= \{ (r_i,\varphi_i,v_i), i\in\N\}$ is a Poisson
process on $\Omega:=\R\times[0,2\uppi)\times[0,1]$ with intensity
$\mu(\mathrm{d}r,\mathrm{d}\varphi,\mathrm{d}v) = \lambda\,\mathrm{d}r\,\mathrm{d}\varphi \,\mathrm{d}v$. This representation
together with the translation invariance of the Poisson line process
readily implies the following result.

%
%pr4.3 #&#
%
\begin{proposition} \label{petaPoissonline}
$\eta$ is a stationary max-i.d. process on $\R^2$.
\end{proposition}

One can construct also max-stable processes of this type. Fix $\alpha
>0$. Start with a Poisson process
$\Phi=\{(r_i,\varphi_i, z_i), i\in\N\}$ on $\mathbb L\times
(0,\infty)$
with intensity $\lambda \,\mathrm{d}r\times \mathrm{d}\varphi\times\alpha z^{-(\alpha+1)}
\,\mathrm{d}z$. Let
$\{F(t), t\in\R\}$ be a process with values in $[0,\infty)$ such that
$\E\int_{\R} F^{\alpha}(r) \,\mathrm{d}r<\infty$. Let $\{F_i, i\in\N\}$ be
independent copies of $F$. Define
%
%e4.7 #&#
%
\begin{equation}
\label{eetaPoissonlinemaxstable} \zeta(x,y)= \max_{i\in\N} z_i
F_i (x\cos\varphi_i+y\sin\varphi_i-r_i).
\end{equation}

%
%pr4.4 #&#
%
\begin{proposition}
$\{\zeta(x,y),(x,y)\in\R^2\}$ is a stationary max-stable process with
$\alpha$-Fr\'echet margins.
\end{proposition}

Max-stability follows directly from the properties of the
Poisson processes and stationarity is the consequence of the
stationarity of
the Poisson line process.
For simplicity, we considered here processes based on Poisson lines in
$\R^2$, but a similar construction is possible in $\R^d$,
where the lines are replaced by $k$-dimensional affine subspaces in $\R
^d, k<d$.
For $k=0$ we recover the mixed moving maxima processes, for $k\geq1$,
however, these processes are generated
by a conservative flow. We show this next for the case of the Poisson
line process ($k=1$).

%
%pr4.5 #&#
%
\begin{proposition} \label{pPoiss-line} The Poisson line max-i.d. process in (\ref{eetaPoissonline}) is generated by a
conservative flow.
\end{proposition}

\begin{pf} Without loss of generality, we may assume that $\operatorname{Leb}[
t\in\bbR\dvtx   F(t)>0] >0$, almost surely.

Indeed, let as above $F(t) = f(t,V)$ and $A:= \{ v\in[0,1]\dvtx \operatorname{Leb}[ t\in\bbR\dvtx   f(t,v)>0] = 0\}$. Suppose first that $\operatorname{Leb}
(A) = 1$, that
is the paths $t\mapsto F(t)$ are zero for almost all $t\in\bbR$, with
probability one. For example, $F(t) = \ind_{B}(t)$ for a set $B$ of
Lebesgue measure zero. In this case, we have that for all fixed
$(x,y)\in\bbR^2$, the random variable $\eta(x,y)$
in~(\ref{eetaPoissonline}) is almost surely zero.\vspace*{-2pt}

On the other hand, if $0<\operatorname{Leb} (A)<1$, consider the process $G(t)=
f(t,W)$, where $W\stackrel{d}{=} V \vert A^c$ have the conditional
distribution of
$V$ restricted to the set $A^c:= [0,1]\setminus A$. By a thinning
argument and replacing in (\ref{eetaPoissonline}), $\lambda$ and $F$ by
$\lambda/\operatorname{Leb}(A^c)$ and $G$, respectively,
we see that for all $(x,y)\in\bbR^2$, we have
\[
\bigl\{\eta(x,y), (x,y)\in\R^2\bigr\} \stackrel{d} {=}  \Bigl\{
\max_{i\in
\N} G_i(x\cos\varphi_i + y
\sin\varphi_i - r_i), (x,y)\in\R^2 \Bigr\}, %
\]
where $G_i$'s are independent copies of $G$. With probability one,
however, the paths of the process $t\mapsto G(t)$ are positive over a
set of positive Lebesgue measure.
This shows that, without loss of generality, we can suppose that (\ref
{eetaPoissonline}) holds with $\operatorname{Leb}[ t\in\bbR\dvtx   F(t)>0] >0$,
almost surely.

Let now $\psi(x)>0, x>0$, be as in Theorem \ref{teoC-int-test}. In
view of (\ref{eeta-x,y-f-rep}), we have that
\[
J_\psi(\varphi,r,v):= \int_{\R^2} \psi\bigl( f(x
\cos\varphi+ y\sin \varphi - r,v) \bigr) \,\mathrm{d}x\,\mathrm{d}y = \int_{\R}  \biggl( \int_{\R} \psi\bigl( f(\widetilde x,v) \bigr) \,\mathrm{d}
\widetilde x  \biggr) \,\mathrm{d}\widetilde y, %
\]
where $\widetilde x:= x\cos\varphi+ y\sin\varphi-r$, and $
\widetilde y:= -x\sin\varphi+ y\cos\varphi$.

Since $\operatorname{Leb}[ t\in\bbR\dvtx   F(t) = f(t,V) >0]>0$ almost surely, we
have $\int_{\R} \psi( f(\widetilde x,V) ) \,\mathrm{d}\widetilde x >0$ a.s. Hence,
$J_\psi(\varphi,r,v) = \infty$ for almost all $(\varphi,r,v)\in\R
\times[0,2\uppi)\times[0,1]$, which means that the
process $\eta$ is generated by a conservative flow.
\end{pf}

%
%re4.6 #&#
%
\begin{remark} In general, we conjecture that the Poisson line process
$X$ for $k\ge1$ is generated by a null-recurrent flow (see,
e.g., Samorodnitsky \cite{samorodnitsky05}, and also \cite
{roy07,kabluchkoextremes,wangstoev10}).
\end{remark}

%s4.5 #&#
\subsection{Penrose min-i.d. random fields} \label{secPenrose}

The next family of examples generalizes the processes considered by
Penrose \cite{penrose88,penrose91,penrose92}. Let $\Pi=\{U_i, i\in\Z
\}$
be the points of Poisson process on $\R^k$ with a constant intensity
$\lambda$. Let $\{\xi_i(t), t\in\R^d\}$, $i\in\Z$, be independent
copies of a random field $\{\xi(t), t\in\R^d\}$ with values in $\R^k$
which has stationary increments. Let $|\cdot|$ be
the Euclidean norm. Define
%
%e4.8 #&#
%
\begin{equation}
\label{eX-Penrose} X(t)=\min_{i\in\Z} \bigl|U_i+
\xi_i(t)\bigr|.
\end{equation}

%
%pr4.7 #&#
%
\begin{proposition}
The process $X$ is stationary, min-i.d. process (i.e., $-X$ is max-i.d.).
\end{proposition}

The min-i.d. property follows directly from the fact that for every
$n\in\N$, we can represent $\Pi$ as a union of $n$ independent Poisson
processes with constant intensity $\frac{\lambda} n$. The
stationarity of $X$ follows from the stationarity of increments of $\xi
$; see Proposition~2.1 in~\cite{kabluchkostationarysystems}. To
construct concrete families of examples one may take
$k=1$ and $\xi$ to be the zero-mean Gaussian process defined on $\R^d$
with covariance function $\E[\xi(t)\xi(s)]=\frac{\sigma^2}2
(|t|^{2H}+ |s|^{2H}- |t-s|^{2H})$, where $H\in(0,1]$ is
the Hurst exponent and $\sigma^2>0$ (see Figure~\ref{figPenrose}). Min-i.d.  processes of this type
appeared in~\cite{kabluchkoextremes2} as limits of pointwise minima (in
the sense of absolute value) of independent
Gaussian processes.

%
%f1 #&#
%
\begin{figure}[b]

\includegraphics{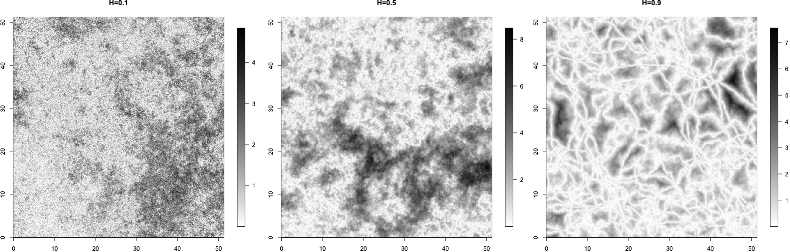}

\caption{Realizations of Penrose-type min-i.d. random fields driven
by isotropic L\'evy fractional Brownian motions defined on $\R^2$ with
Hurst exponents $H = 0.1$, $0.5$, and $0.9$, respectively, left to right.}\label{figPenrose}
\end{figure}

One can also take $d=1$, $k\in\bbN$ arbitrary and let $\xi$ be the
$\R
^k$-valued standard Brownian motion. The next result shows that the
resulting processes, which were introduced and studied by Penrose~\cite
{penrose88,penrose91,penrose92}, are of mixed moving maximum
type for $k\geq3$ and are conservative for $k\leq2$.

%
%pr4.8 #&#
%
\begin{proposition} Let $X=\{X(t), t\in\R\}$ be as in (\ref
{eX-Penrose}), where $\{\xi(t), t\in\R\}$ is the standard Brownian
motion in $\R^k$.
The max-i.d. process $-X$ is generated by a conservative flow for
$k=1, 2$ and dissipative for $k\geq3$.
\end{proposition}

\begin{pf}
We shall apply the integral test in Theorem \ref{teoC-int-test} above.
In the case $k=1,2$ the result follows from the
neighborhood-recurrence property of the Brownian motion. In the case
$k\geq3$ we will use the fact that the Brownian
motion in $\R^k$ is transient.

Consider the space $\Omega:= \R\times C_0(\R,\R^k)$,
equipped with the product of the Borel $\sigma$-algebras, where
$C_0(\R,\R^k)$ is the space of $\R^k$-valued continuous functions on $\R$
which vanish at $0$. Consider the Poisson point process
$\Pi= \{ (U_i,\xi_i), i\in\mathbb Z\}$ on $\Omega$ with intensity
$\mu(\mathrm{d}u,\mathrm{d}v) = \lambda(\mathrm{d}u) \P_\xi(\mathrm{d}v)$,
where $\lambda$ is the Lebesgue measure on $R$ and $\P_\xi$ is the law
of $\xi$ on $C_0(\R,\R^k)$. Therefore, for the max-i.d. process
$-X$ we obtain the spectral representation
\[
-X \stackrel{d} {=}  \biggl\{ \int_{\Omega}^\vee
f_t(u,v) \Pi_\mu (\mathrm{d}u,\mathrm{d}v) \biggr\}_{t\in\R},
\qquad\mbox{with } f_t (u,v) = -\bigl|u+ v(t)\bigr|, %
\]
where $v = (v(t))_{t\in\R}\in C_0(\R,\R^k)$. Observe that since $\P
_\xi
[ v(0)= 0] = 1$, we have $\int_\Omega f_0 \,\mathrm{d}\mu= \int_{\bbR} \mathrm{e}^{-|u|}
\,\mathrm{d}u <\infty$ and
one can take $\psi(x):= \mathrm{e}^{x}$ in (\ref{eC-test}).

\emph{Consider first the transient case $k\geq3$.} By the
Dvoretzky--Erd\"os criterion (see, e.g., Theorem~3.22 in \cite
{mortersperes2010}) by taking $g(r) = r^{1/3}$,
we obtain that $\int_1^\infty g(r)^{k-2} r^{-k/2} \,\mathrm{d}r <\infty$, and
therefore $\liminf_{|t|\to\infty} |\xi(t)|/g(t) = \infty$, with
probability one. Thus, for $\P_\xi$-almost all $v$,
we have
\[
\psi\bigl(f_t (u,v)\bigr) = \mathrm{e}^{-|u+v(t)|} \leq\exp\bigl
\{-|t|^{1/3}\bigr\}
\]
for all sufficiently large $|t|$. Since the latter bound is integrable,
we obtain that $\int_\bbR\psi(f_t (u,v)) \,\mathrm{d}t <\infty$ for $\mu$-almost
all $(u,v)\in\Omega$. This, in view of Theorem \ref{teoC-int-test}
implies that $-X$ is generated by a dissipative flow.

\emph{Suppose now $k\leq2$.} By the neighborhood recurrence of the
Brownian motion~\cite{mortersperes2010} in dimensions $k=1,2$, the time
which the Brownian motion spends in any open set is infinite with
probability $1$. It follows immediately that $\int_\bbR\psi(f_t (u,v))
\,\mathrm{d}t =\infty$ for $\mu$-almost every $(u,v)\in\Omega$. By Theorem~\ref{teoC-int-test} this
implies that $-X$ is generated by a conservative flow.
\end{pf}

%s4.6 #&#
\subsection{Stationary union-i.d. random sets}\label{secU-id}
A measurable process $\{X(t), t\in\R^d\}$ taking only values $0,1$ can
be identified with the random set $S:=\{t\in\R^d\dvtx  X(t)=1\}$. Note that
we do not require the sets to be, say, closed. If the process $X$ is
max-i.d., then the corresponding random set $S$ is \textit
{union-i.d.} and vice versa. This means that for every $n\in\N$ we can
find i.i.d. random sets $C_1,\ldots,C_n$ such that $S$ has the same
finite-dimensional distributions as $C_1\cup\cdots\cup C_n$;
see~\cite
{molchanovbook}, Chapter~4. The process $X$ is stochastically
continuous iff the random set $S$ is stochastically continuous in the
following sense: for every \mbox{$t\in\R^d$,} $\lim_{s\to t}\P[t\in S,
s\notin
S]=\lim_{s\to t}\P[t\notin S, s\in S]=0$. Using Theorem~\ref
{theoflowrepexists}, we can describe all stationary stochastically
continuous union-i.d. random sets.

%th4.9 #&#
%
\begin{theorem}\label{theoflowreprandomsets}
Let $S$ be a stationary, stochastically continuous, union-i.d. random
set in $\R^d$. Then there is a $\sigma$-finite Borel space $(\Omega,
\BB,\mu)$, a measurable, measure-preserving $\R^d$-action $\{T_t\}
_{t\in\R^d}$ on $(\Omega,\BB,\mu)$, and a set $A\in\BB$ with
$\mu
(A)<\infty$ such that
%
%e4.9 #&#
%
\begin{equation}
\label{eqspecrepstatset} S \stackrel{d} {=} \bigl\{t\in\R^d\dvtx  \Pi_{\mu}
\bigl(T_t^{-1}(A)\bigr) \neq 0 \bigr\},
\end{equation}
where $\Pi_{\mu}$ is a Poisson random measure on $(\Omega,\mathcal B)$
with intensity measure $\mu$.
\end{theorem}

\begin{pf}
Let $\{X(t), t\in\R^d\}$ be the $\{0,1\}$-valued max-i.d.~process
corresponding to $S$, that is, $X(t)=\ind_{t\in S}$. Then, $X$ has a
flow representation in the sense of
Theorem~\ref{theoflowrepexists}. The function $f_0$ in this
representation takes only values $0,1$, ($\mod \mu$). That is, $f=\ind
_A$ for some set $A\in\BB$. Note that $A$ has finite
measure since $f\in\LL^{\vee}$. Since $f_0\circ T_t = \ind
_{T_t^{-1} A
}$, the statement of the theorem follows.
\end{pf}

%
%ex4.10 #&#
%
\begin{example}
Let $\Omega$ be the space $\R^d$ endowed with the Lebesgue measure.
Consider a flow $T_t(\omega)=\omega-t, \omega,t\in\R^d$. Let
$A\subset\R^d$ be a Borel set of
finite measure and let $\Pi= \{U_i, i\in\N\}$ be a unit intensity
Poisson process on $\R^d$. Then, the corresponding union-i.d. stationary random set has the form $S = \bigcup_{i\in\N} (U_i - A)$
and it is known in the literature as the Boolean model with
(non-random) grain $A$. More generally, one can let $\Omega:= \R^d
\times E$ be the product of $\R^d$ with some probability
space $E$ and define $T_t(x,y):= (x-t,y), x,t\in\R^d, y\in E$ as the
shift of the first coordinate. By taking a random set $A = A(y)$, one obtains
a \emph{mixed} or random grain Boolean model, which is similar to and in
fact corresponds to the level-set of a mixed moving maxima random
field model.
\end{example}

%s5 #&#
\section{Proofs} \label{secproofs}
%s5.1 #&#
\subsection{Lemma on conjugacy between collections of functions}
The following lemma is used in the proofs of Theorems~\ref
{theominspecrepmaxunique} and~\ref{theominspecrepsumunique}.

%le5.1 #&#
%
\begin{lemma} \label{lft-point}
Let $(\Omega_i,{\BB}_i,\mu_i), i=1,2$ be two measure spaces. Consider
two families of measurable functions
$f_t^{(i)}\dvtx \Omega_i\to\R, t\in T$, $i=1,2$, and define two measurable
mappings $F_i\dvtx  (\Omega_i, \BB_i)\to(\R^T, \BB)$ by
$
F_i(\omega)=(f^{(i)}_t(\omega))_{t\in T}
$,
$\omega\in\Omega_i$,
$i=1,2$. Here, $\BB$ is the product \mbox{$\sigma$-}algebra on $\R^T$.
Assume that
\begin{enumerate}[2.]
\item $\sigma\{ f_t^{(i)}, t\in T\}={\BB}_i\mod\mu_i$, $i=1,2$.
\item The induced measures $\mu_1\circ F_1^{-1}$ and $\mu_2\circ
F_2^{-1}$ are equal on $(\R^T,\BB)$.
\end{enumerate}

Then, the following two claims are true:
\begin{enumerate}[(ii)]
\item[(i)] If $(\Omega_1,\BB_1)$ is a Borel space, then there exists a
measurable map $\Phi\dvtx \Omega_2 \to\Omega_1$
such that $\mu_1 = \mu_2 \circ\Phi^{-1}$ and for all $t\in T$, we have
$f_t^{(2)} = f_t^{(1)}\circ  \Phi$, $\mu_1$-a.e.

\item[(ii)] If both $(\Omega_i,\BB_i), i=1,2$ are Borel spaces, then
the mapping in part \textup{(i)} is a measure space
isomorphism and it is unique (modulo null sets).
\end{enumerate}
\end{lemma}

\begin{pf}
It will be convenient to identify the sets in ${\BB}_i$ that are equal
modulo $\mu_i, i=1,2$.
Formally, let ${\mathcal I}_i\subset{\BB}_i$ be the $\sigma$-ideals
of $\mu_i$-null sets in the spaces
$(\Omega_i,{\BB}_i,\mu_i)$ (see, e.g., Chapter~II.21 in~\cite
{sikorskibook}) and let
$[{\BB}_i]:= {\BB}_i/{\mathcal I}_i$ be the corresponding factor
$\sigma
$-fields, $i=1,2$. The elements of $[\BB_i]$ are the equivalence classes
$[B]=\{A\in\BB_i\dvtx  \mu_i(A\Delta B)=0\}$, where $B\in\BB_i, i=1,2$.

We shall define next a $\sigma$-isomorphism $U\dvtx  [{\BB}_1] \to[{\BB
}_2]$, that is, a bijective mapping that preserves countable unions and complements.
For all $B\in{\BB}_1$, we set
%
%e5.1 #&#
%
\begin{equation}
\label{eqdefbaru} U\bigl([B]\bigr):= \bigl[F_2^{-1}(A)\bigr],
\qquad\mbox{where } \bigl[F_1^{-1}(A)\bigr] = [B].
\end{equation}
Note that such an $A \in{\BB}$ exists since by assumption
$F_1^{-1}({\BB}) =
\sigma\{ f_t^{(1)}, t\in T\} = {\BB}_1 \mod\mu_1$. One can readily
see that the mapping $U$ is a well-defined $\sigma$-isomorphism.
Indeed, since $\mu_1\circ F_1^{-1} = \mu_2\circ F_2^{-1}$, for every
$A',A''\in\BB$,
\begin{eqnarray*}
\mu_1\bigl(F_1^{-1}\bigl(A'
\bigr) \Delta F_1^{-1}\bigl(A''
\bigr)\bigr) &=& \mu_1\bigl(F_1^{-1}
\bigl(A' \Delta A''\bigr)\bigr)
\\
&=&
\mu_2\bigl(F_2^{-1}\bigl(A' \Delta
A''\bigr)\bigr) = \mu_2
\bigl(F_2^{-1}\bigl(A'\bigr) \Delta
F_2^{-1}\bigl(A''\bigr)\bigr).
\end{eqnarray*}
Thus, $F_1^{-1}(A') = F_1^{-1}(A'') \mod \mu_1$,
if and only if $F_2^{-1}(A') = F_2^{-1}(A'') \mod \mu_2$, and the
definition of $U$ does not depend on the
choice of the representative $B$ of the equivalence class $[B]$ and on
the choice of $A$ in~(\ref{eqdefbaru}). This shows, moreover, that
$[B'] = [B'']$
if and only if $U ([B']) = U([B''])$, that is, $U$ is \emph{injective}.
On the other hand, since
$F_2^{-1}({\BB}) = \sigma\{ f_t^{(2)}, t\in T\} = {\BB}_2\mod\mu_2$,
for all $B \in{\BB}_2$, we have $[F_2^{-1}(A)] = [B]$, for
some $A\in{\BB}$ and hence $U ([F_1^{-1}(A)]) = [B]$. This shows that
$U$ is \emph{onto} and hence a bijection. Also, since $\mu
_1(F_1^{-1}(A))=\mu_2(F_2^{-1}(A))$, we have by~(\ref{eqdefbaru}) that
$U$ is measure-preserving. Since
$U$ clearly preserves the countable unions and complements, it is a
$\sigma$-isomorphism.

\emph{Under the assumption of part \textup{(i)}, we have that $(\Omega_1,{\BB
}_1)$ is a Borel space.} Then, Theorem 32.5 of~\cite{sikorskibook}
implies that the
$\sigma$-isomorphism $U$ is induced by a measurable point mapping
$\Phi\dvtx \Omega_2\to\Omega_1$ in the following sense:
%
%e5.2 #&#
%
\begin{equation}
\label{equbaru} U\bigl([B]\bigr) = \bigl[\Phi^{-1}(B)\bigr],\qquad B\in
\mathcal B_1.
\end{equation}
Clearly, since $U$ is a $\sigma$-isomorphism, we also have that $\mu_1
= \mu_2 \circ\Phi^{-1}$.

Let us fix some $t\in T$ and show that $f_t^{(2)}= f_t^{(1)}\circ
\Phi
$ holds $\mu_2$-a.e.
Let $I$ be a Borel subset of $\R$ and consider the cylinder set
$A = \{ \varphi\dvtx  T \to\R\dvtx  \varphi(t)\in I\}\subset\R^T$. We have
%
%e5.3 #&#
%
\begin{eqnarray}\label{eqequalpreimages}
\bigl[\bigl(f_t^{(1)}\circ\Phi\bigr)^{-1}(I)
\bigr] &=& \bigl[\Phi^{-1}\bigl(\bigl(f_t^{(1)}
\bigr)^{-1}(I)\bigr)\bigr]= U\bigl(\bigl[\bigl(f_t^{(1)}
\bigr)^{-1}(I)\bigr]\bigr)
\nonumber\\[-8pt]\\[-8pt]
&=& U \bigl(\bigl[F_1^{-1}(A)\bigr]\bigr)=
\bigl[F_2^{-1}(A)\bigr]=\bigl[\bigl(f_t^{(2)}
\bigr)^{-1}(I)\bigr].\nonumber
\end{eqnarray}
Assume that $f_t^{(2)}\neq f_t^{(1)}\circ\Phi$ on $D\in{\BB}_2$ with
$\mu_2(D)>0$. Then we can find an
$\varepsilon>0$ and a measurable set $D'\subset D$ with $\mu_2(D')>0$ such
that $|f_t^{(2)}-f_t^{(1)}\circ  \Phi|>\varepsilon$
everywhere on $D'$. Further, we can find a $k\in\Z$ and a measurable
set $D''\subset D'$ with $\mu_1(D'')>0$
such that with $I=[k\varepsilon, (k+1)\varepsilon)$, we have
$f_t^{(2)}\in I$
everywhere on $D''$. It then follows that
$f_t^{(1)}\circ\Phi\notin I$ on $D''$. But this contradicts~(\ref
{eqequalpreimages}), which implies
$f_t^{(2)} = f_t^{(1)} \circ\Phi$, $\mu_1$-a.e.

\emph{Now, we turn to proving part \textup{(ii)}.} That is, that $\Phi$ a measure
space isomorphism and unique (modulo null sets) under the additional assumption
that $(\Omega_2,\BB_2)$ is a Borel space. By applying the above
argument to the $\sigma$-isomorphism $U^{-1}\dvtx [\BB_2] \to[\BB_1]$, we
obtain that
there exists a measurable, measure-preserving $\widetilde\Phi\dvtx \Omega_1
\to\Omega_2$, such that
\[
U^{-1}\bigl([B]\bigr) = \bigl[\widetilde\Phi^{-1} (B)\bigr]
\qquad\mbox{for all }B\in\BB_2. %
\]
Therefore, $\Psi:= \Phi\circ\widetilde\Phi\dvtx \Omega_1 \to\Omega
_1$ is
measurable and since $U\circ U^{-1} \equiv{\rm id}$, we have that
$[\Psi(A)]= [A]$ for all $A\in{\BB}_1$. We will use the fact that
$(\Omega_1,{\BB}_1)$ is a Borel space to show that $\Psi= {\rm id}
\mod\mu_1$, which will
imply that $\Phi$ is a measure space isomorphism (Definition \ref
{defisomorphisms}).

By Kuratowski's theorem, $(\Omega_1,{\BB}_1)$ is isomorphic to either
$(E,2^E)$, where $E$ is an at most countable set, or $(\bbR, {\BB
}_\bbR
)$ -- the
real line equipped with the Borel $\sigma$-algebra. The discrete case
is trivial. Suppose now the latter is true and without loss of generality
let $(\Omega_1,{\BB}_1) \equiv(\bbR, {\BB}_\bbR)$. Let
$\varepsilon>0$ be
arbitrary and suppose that
$\mu_1(\{|\Psi- {\rm id}|>\varepsilon\}) >0$, then for some $k\in
\bbZ$, we
have for
$D:=\{|\Psi- {\rm id}|>\varepsilon\}\cap[k\varepsilon,
(k+1)\varepsilon)$ that $\mu_1(D)
>0$. But then
$\Psi(x)\notin[k\varepsilon, (k+1)\varepsilon)$, for all \mbox{$x\in D$},
and hence
$\Psi(D) \cap D =\varnothing$. This contradicts the fact that $[\Psi(D)]
= [D]$ because \mbox{$\mu_1(D)>0$}. Since
$\varepsilon>0$ was arbitrary, it follows that $\Psi= {\rm id} \mod
\mu_1$
and hence $\Phi^{-1} = \widetilde\Phi \mod\mu_1$.

To complete the proof, we need to show the uniqueness of $\Phi$. Assume
that $\Phi_*\dvtx \Omega_2\to\Omega_1$ is another measure space isomorphism
such that for all $t\in T$, $f_t^{(2)}=f_t^{(1)}\circ\Phi_*$, $\mu
_2$-a.e. Then, relation (\ref{eqequalpreimages}) holds with $\Phi$
replaced by $\Phi_*$,
which implies that $\Phi_*$ induces the same $\sigma$-isomorphism $U$ as
$\Phi$. Since $\Phi_*$ is a measure space\vspace*{1pt} isomorphism, the measurable
map $\widetilde\Phi:= (\Phi_*)^{-1}$ induces the $\sigma$-isomorphism
$U^{-1}$ and hence $\Psi:= \Phi\circ\widetilde\Phi$ induces the
identity $\sigma$-isomorphism on the Borel space $(\Omega_1,\BB_1)$.
As argued above, this implies that $\Phi\circ\widetilde\Phi= {\rm
id}$, ($\mod \mu_1$).
\end{pf}

%s5.2 #&#
\subsection{Proofs in the max-i.d. case}
\begin{pf*}{Proof of Theorem~\ref{theominspecrepmaxexists}}
Write $\R_+=[0,\infty)$. Let $T_0$ be the at most countable set
appearing in Condition~\textup{S}. Let $\R_+^{T_0}$ be the space of functions
$\varphi\dvtx T_0\to\R_+$ endowed with the product \mbox{$\sigma$-}algebra~$\BB$.
Denote by $\nu$ the exponent measure of the process $\{X(t), t\in T_0\}
$; see Vatan \cite{vatan85}. It is a $\sigma$-finite measure on $\R
_+^{T_0}$ such that for every $t_1,\ldots,t_n\in T_0$ and $x_1,\ldots,x_n>0$ we have
%
%e5.4 #&#
%
\begin{equation}
\label{eqmaxidspecmeasurefidi} \P\bigl\{ X(t_j) < x_j, 1\leq j\leq n
\bigr\} = \exp \Biggl\{ -\nu \Biggl( \bigcup_{j=1}^n
\bigl\{\varphi\in\R_+^{T_0}\dvtx  \varphi(t_j) \geq
x_j\bigr\} \Biggr)  \Biggr\}.
\end{equation}
We\vspace*{1pt} agree that $\nu(\{0\})=0$ (which is different from~\cite{vatan85}).
Taking the coordinate mappings $f_t\dvtx \R_+^{T_0}\to\R$, $f_t(\varphi
)=\varphi(t)$, $t\in T_0$, we therefore obtain a spectral
representation of $\{X(t), t\in T_0\}$ on $(\R_+^{T_0}, \BB)$. To see
this, compare~(\ref{eX-fdd}) and~(\ref{eqmaxidspecmeasurefidi}). Let
$t\in T$ be arbitrary. Condition~\textup{S} states that there exists a
sequence $\{t_n\}_{n\in\N}\subset T_0$ such that $X(t_n)\to X(t)$ in
probability. Thus, the sequence $X(t_n)$ is Cauchy in probability. By
the equality of the finite-dimensional distributions, the sequence
$I(f_{t_n})$ is Cauchy in probability, and therefore, it converges in
probability. By Theorem~4.5 in~\cite{balkemaetal93}, there is a
function $f_t\in\LL^{\vee}(\R_+^{T_0},\BB,\nu)$ such that $I(f_{t_n})$
converges in probability to $I(f_t)$.
Theorem~4.4 of~\cite{balkemaetal93} implies that
the finite-dimensional distributions of $\{I(f_t), t\in T\}$ and $\{
X(t), t\in T\}$ are equal, that is, the collection $\{f_t, t\in T\}$
is a spectral representation of $\{X(t), t\in T\}$ on $(\R_+^{T_0},\BB
)$. Since the coordinate functions $f_t$, $t\in T_0$, generate the
product \mbox{$\sigma$-}algebra $\BB$, and $\nu(\bigcap_{t\in T_0}\{f_t=0\}
)=\nu
(\{0\})=0$, this representation is minimal. To complete the proof note
that by Kuratowski's theorem, for at most countable $T_0$, the
measurable space $(\R_+^{T_0},\BB)$ is isomorphic to $[0,1]$ endowed
with the Borel $\sigma$-algebra.
\end{pf*}

\begin{pf*}{Proof of Theorem~\ref{theominspecrepmaxunique}}
As in Lemma~\ref{lft-point}, we define two measurable mappings $F_i\dvtx
(\Omega_i, \BB_i)\to(\R^T, \BB)$
by
\[
F_i(\omega)=\bigl(f^{(i)}_t(\omega)
\bigr)_{t\in T},\qquad \omega\in\Omega_i, i=1,2. %
\]
The\vspace*{1pt} first condition of Lemma~\ref{lft-point} is satisfied by the
assumption of minimality. We will show that the induced measures $\mu
_1\circ F_1^{-1}$ and $\mu_2\circ F_2^{-1}$ are equal on $(\R^T,\BB)$.
We will prove that for all $t_1,\ldots,t_n\in T$ and all intervals
$[x_1, y_1),\ldots, [x_n,y_n)\subset\R$ we have
%
%e5.5 #&#
%
\begin{equation}
\label{eqmu1mu2eqfidi} \mu_1 \Biggl( \bigcap_{j=1}^n
\bigl\{x_j\leq f_{t_j}^{(1)}< y_j
\bigr\} \Biggr) = \mu_2 \Biggl( \bigcap_{j=1}^n
\bigl\{x_j\leq f_{t_j}^{(2)} < y_j
\bigr\} \Biggr).
\end{equation}
Recall that $\{f_t^{(i)}, t\in T\}$, $i=1,2$, are spectral
representations of the same process $X$. By~(\ref{eX-fdd}), we have
that for all $x_1,\ldots,x_n>0$,
%
%e5.6 #&#
%
\begin{equation}
\label{etheominspecrep-1} \mu_1 \Biggl( \bigcup_{j=1}^n
\bigl\{ f_{t_j}^{(1)} \geq x_j \bigr\} \Biggr) =
\mu_2 \Biggl( \bigcup_{j=1}^n
\bigl\{ f_{t_j}^{(2)} \geq x_j \bigr\} \Biggr).
\end{equation}
Note that $\mu_i(\{f_t^{(i)} >x\}) <\infty$ for all $x>0$, $t\in T$,
since $f_t^{(i)}\in\LL^{\vee}$. Using this fact and the
inclusion--exclusion formula, we obtain that
relation (\ref{etheominspecrep-1}) is also valid with the unions
therein replaced by intersections. This proves that~(\ref
{eqmu1mu2eqfidi}) holds provided that $0<x_j < y_j$ for all $j=1,\ldots,
n$. Note that this argument breaks down if $x_j=0$ for some $j$ since
we cannot apply the inclusion--exclusion formula to sets of infinite
measure. To show that the measures $\mu_1\circ F_1^{-1}$ and $\mu
_2\circ F_2^{-1}$ agree on the ``boundary'' of $\R_+^T$ we need a
separate argument.

We now show that~(\ref{eqmu1mu2eqfidi}) continues to hold even if some
of the $x_j$'s are allowed to be zero. We do not need to consider the
case of negative $x_j$'s since $f_t^{(i)}\geq0$, $\mu_i$-a.e., by
definition of $\LL^{\vee}$. By letting some of the $x_j$'s go to $0$
and using continuity of measure we obtain that~(\ref{eqmu1mu2eqfidi})
continues to hold if some of the sets of the form $\{x_j\leq
f_{t_j}^{(i)}< y_j\}$ therein\vspace*{1pt} are replaced by $\{0<f_{t_j}^{(i)}< y_j\}
$. By additivity of measure, the proof of~(\ref{eqmu1mu2eqfidi}) in
full generality will be completed if we show that~(\ref
{eqmu1mu2eqfidi}) continues to hold if some of the sets of the form $\{
x_j\leq f_{t_j}^{(i)}< y_j\}$ therein are replaced by $\{
f_{t_j}^{(i)}=0\}$. Let us make this statement precise. Take $l,m\in\N
_0$, $s_1,\ldots,s_l\in T$, $r_1,\ldots,r_m\in T$ and $0<u_1<v_1,
\ldots, 0<u_m<v_m$. Define two measurable sets $C_i\subset\Omega_i$,
$i=1,2$, by
\[
C_i=A_i\cap B_i,\qquad A_i=
\bigcap_{k=1}^l \bigl
\{f_{s_k}^{(i)}=0\bigr\},\qquad B_i=\bigcap
_{j=1}^m \bigl\{u_j\leq
f_{r_j}^{(i)}<v_j \bigr\},\qquad i=1,2.
\]
We will show that $\mu_1(C_1)=\mu_2(C_2)$. Suppose first that $m\neq
0$. Then, $\mu_i(C_i)=\mu_i(B_i)-\mu_i(B_i\cap D_i)$, where
\[
D_{i}=\bigcup_{k=1}^l \bigl
\{f_{s_k}^{(i)}>0\bigr\} = \bigcup
_{n\in\N}D_{i,n},\qquad D_{i,n}=\bigcup
_{k=1}^l \biggl\{\frac{1}n\leq
f_{s_k}^{(i)} < n \biggr\},\qquad i=1,2. %
\]
We have already shown that~(\ref{eqmu1mu2eqfidi}) holds if $x_j>0$ for
all $j=1,\ldots,n$. This implies that $\mu_1(B_1)=\mu_2(B_2)$ (where
both terms are finite since $m\neq0$). Also, by the
inclusion--exclusion formula, $\mu_1(B_1\cap D_{1,n})=\mu_{2}(B_2\cap
D_{2,n})$ for every $n\in\N$. Note that $D_{i,1}\subset
D_{i,2}\subset \cdots.$ Letting $n\to\infty$ and using the continuity of measure, we
obtain $\mu(B_1\cap D_1)=\mu_2(B_2\cap D_2)$. This proves that $\mu
_1(C_1)=\mu_2(C_2)$ in the case $m\neq0$.

Consider now the case $m=0$. In this case it is possible that $\mu
_i(B_i)=\infty$ and the above argument breaks down. We show that $\mu
_1(C_1)=\mu_2(C_2)$, or, equivalently, $\mu_1(A_1)=\mu_2(A_2)$.
We will use the minimality and an exhaustion argument (cf. Lemma~1.0.7
in~\cite{aaronsonbook}) to show that there is a sequence
$q_1,q_2,\ldots \in T$ such that
%
%e5.7 #&#
%
\begin{equation}
\label{equnionsuppomegai} \mu_i \biggl(\bigcap_{n\in\N}
\bigl\{f_{q_n}^{(i)}=0\bigr\} \biggr)=0,\qquad i=1,2.
\end{equation}
Fix $i\in\{1,2\}$. Since the measure $\mu_i$ is $\sigma$-finite, we
can represent $\Omega_i$ as a disjoint union of sets $E_1,E_2,\ldots
\in
\BB_i$ such that $\mu_i(E_k)<\infty$, $k\in\N$. Let $e_k=\inf_{Q}\mu
_i(\bigcap_{q\in Q} \{f_{q}^{(i)}=0\}\cap E_k)$, where the infimum is
taken over all at most countable sets $Q\subset T$. Clearly,
$e_k<\infty
$. For every $n\in\N$ we can find at most countable $Q_{kn}\subset T$
such that $\mu_i(\bigcap_{q\in Q_{kn}} \{f_{q}^{(i)}=0\}\cap
E_k)<e_k+\frac{1}n$. Since $Q_k:=\bigcup_{n\in\N} Q_{kn}$ is at most
countable, we have
$
e_k=\mu_i(F_k)
$, where $F_k=\bigcap_{q\in Q_k} \{f_{q}^{(i)}=0\}\cap E_k$.
It follows that for every $t\in T$, $f_t^{(i)}=0$ a.e. on $F_k$.
Otherwise, we could consider $Q_k\cup\{t\}$ and arrive at a contradiction.
By the assumption of minimality this implies that, we must have
$e_k=0$. This holds for every $k\in\N$. The proof of~(\ref
{equnionsuppomegai}) is completed by taking the union of the collections
$Q_k$, $k\in\N$.

Consider measurable sets
\[
G_{i,p}=A_i \cap \Biggl(\bigcap
_{k=1}^{p-1} \bigl\{f_{q_k}^{(i)}=0
\bigr\} \Biggr) \cap \bigl\{f_{q_p}^{(i)}>0\bigr\},\qquad p\in
\N, i=1,2. %
\]
We have $\mu_1(G_{1,p})=\mu_2(G_{2,p})$ for every $p\in\N$. Indeed, by
continuity of measure,
\[
\mu_i(G_{i,p})=\lim_{n\to\infty}
\mu_i \Biggl(A_i \cap \Biggl(\bigcap
_{k=1}^{p-1} \bigl\{f_{q_k}^{(i)}=0
\bigr\} \Biggr) \cap \biggl\{\frac{1}n \leq f_{q_p}^{(i)}<n
\biggr\} \Biggr). %
\]
The right-hand side does not depend on $i=1,2$ as a particular case of
$\mu_1(C_1)=\mu_2(C_2)$ in the case $m>0$.
It follows from~(\ref{equnionsuppomegai}) that
\[
\mu_1(A_1) = \sum_{p=1}^{\infty}
\mu_1(G_{1,p}) = \sum_{p=1}^{\infty}
\mu_2(G_{2,p}) = \mu_2(A_2).
\]
This completes the proof of~(\ref{eqmu1mu2eqfidi}).

It follows now from~(\ref{eqmu1mu2eqfidi}) that the measures $\mu
_1\circ F_1^{-1}$ and $\mu_2\circ F_2^{-1}$ coincide on the semi-ring
$\mathcal C$ consisting of sets of the form
$
\bigcap_{j=1}^n \{\varphi\dvtx T\to\R\dvtx  x_j\leq\varphi(t_j)< y_j\}$,
where $t_1,\ldots,t_n\in T$, $[x_1, y_1), \ldots, [x_n,y_n)\subset\R$.
Note that $\mathcal C$ generates the product $\sigma$-algebra $\BB$.
Also, by~(\ref{equnionsuppomegai}), we can represent $\R^T$ as
\[
\R^T = \bigcup_{n=1}^{\infty}
\bigcup_{k=1}^{\infty} \bigl\{\varphi\dvtx T\to\R\dvtx
k^{-1} \leq\varphi(q_n) < k \bigr\} \mod \mu _1
\circ F_1^{-1}\mbox{ and }\mu_2 \circ
F_2^{-1}. %
\]
Note that the sets in the union in the right-hand side have finite $\mu
_1\circ F_1^{-1}$ (and $\mu_2\circ F_2^{-1}$) measure and belong to the
semiring $\mathcal C$. The uniqueness of the extension of measure
theorem yields that $\mu_1\circ F_1^{-1}=\mu_2\circ F_2^{-1}$. The
assumptions of Lemma~\ref{lft-point} are verified. Lemma \ref
{lft-point} yields~(\ref{eqtheominspecrep}) and completes the proof of
the theorem.
\end{pf*}

%s5.3 #&#
\subsection{Proofs in the i.d. case}

We start with discussing some properties of the spectral representation.
Note first that the functional $I$ is not additive. Nevertheless, by
(\ref{ejoint-chf}) it follows that for all $f,g\in\LL^+$
%
%e5.8 #&#
%
\begin{eqnarray}\label{egamma-def}
I(f) + I(g) &=& I(f+g) + \gamma(f,g),\qquad\mbox{where}
\nonumber\\[-8pt]\\[-8pt]
\gamma(f,g)&:=&
\int_\Omega \bigl(a(f+g) - a(f) - a(g)  \bigr) \,\mathrm{d}\mu.\nonumber
\end{eqnarray}

The next result shows that $\gamma$ in (\ref{egamma-def}) is well
defined and
that this constant correction term can be controlled in terms of the
metric $d$.

%
%le5.2 #&#
%
\begin{lemma}\label{lgamma} For all $f, g \in\LL^+$, we have that
$\int_\Omega| a(f) + a(g) -a(f+g)| \,\mathrm{d}\mu<\infty$.
Moreover, for $\gamma$ and $d$ as in (\ref{egamma-def}) and (\ref
{ed-def}), we have
\[
\bigl|\gamma(f,g)\bigr| \le3 \bigl(d(f+g)\bigr)^2 + 2 \bigl(d(f) + d(g)\bigr)
d(f+g). %
\]
\end{lemma}

\begin{pf}
Consider the integral defining $\gamma(f,g)$ over the sets $A:= \{
|f+g|>1 \}$,
$B:=A^c\cap \{ |f| \le1\} \cap\{|g|\le1\}$ and $C:=A^c \cap(\{ |f|
>1\} \cup\{ |g| >1\})$, which form a
disjoint partition of $\Omega$.

Observe that over $B$ the integrand is zero, since $a(f) = f, a(g) =
g$ and $a(f+g) = f+g$ whenever
$|f|\le1, |g|\le1$, and $|f+g|\le1$. Note, on the other hand that
the set
$B^c = A \cup C \subset\{|f+g| >1\} \cup\{|f|>1\} \cup\{|g|>1\}$ has
a finite $\mu$ measure because
$f, g$ and $f+g$ belong to $\LL^+$. Since $|a(f+g)-a(f) - a(g)| \le3$
it therefore follows that
$\int_{\Omega}|a(f+g)-a(f) - a(g)| \,\mathrm{d}\mu<\infty$ and $\gamma(f,g)$ is
well defined.

Over $A$, we have that
\[
\int_{\{|f+g|>1\}} \bigl|a(f+g)-a(f) - a(g)\bigr| \,\mathrm{d}\mu\le3 \mu\bigl\{ |f+g| >1\bigr\}
\le 3 \bigl(d(f+g)\bigr)^2. %
\]
Now, focus on the set $C$. The function $a$ in (\ref{ea-def}) is
Lipschitz and in fact $|a(x) + a(y)| \le|x+y|$ for all
$x,y\in\R$. Therefore, $|a(f+g)-a(f) - a(g)| \le2 |f+g|$, and hence
\begin{eqnarray*}
\int_{C} \bigl|a(f+g)-a(f) - a(g)\bigr| \,\mathrm{d}\mu &\le& 2 \int
_{\{ |f+g| \le1\} } |f+g| (\ind_{\{|f|>1\}} + \ind_{\{|g|>1\}}) \,\mathrm{d}
\mu
\\
&\le& 2 d(f+g) \bigl(d(f) + d(g)\bigr),
\end{eqnarray*}
where the last relation follows from the Cauchy--Schwartz inequality
and the fact that $\mu\{|f|>1\} \le d(f)^2$. Combining the above two
bounds, we obtain the desired inequality.
\end{pf}

Relation (\ref{egamma-def}) readily implies the following result
on the sum of two spectral representations over the same space.

%
%pr5.3 #&#
%
\begin{proposition}
Consider two i.d. processes $X_t^{(i)}:= I(f_t^{(i)}) + c_t^{(i)}, t\in T$,
where $\{f_t^{(i)}\}_{t\in T}\subset\LL^+(\Omega, \BB, \mu)$, and
$c_t^{(i)} \in\R, i=1,2$. Then, their sum
has the following spectral representation:
\[
\bigl\{ X_t^{(1)} + X_t^{(2)} \bigr
\}_{t\in T} \stackrel{d} {=} \bigl\{I\bigl(f_t^{(1)} +
f_t^{(2)}\bigr)+c_t^{(1)}
+c_t^{(2)} + \gamma\bigl(f_t^{(1)},f_t^{(2)}
\bigr), t\in T \bigr\}. %
\]
\end{proposition}

Recall now that convergence in probability is metrized by the Ky Fan
distance which is given by
%
%e5.9 #&#
%
\begin{equation}
\label{eKyFan} d_{\mathrm{KF}}(\xi,\eta) \equiv d_{\mathrm{KF}}(\xi-\eta):=
\inf\bigl\{ \delta>0\dvtx  \P \bigl\{ |\xi-\eta| \ge\delta\bigr\} \le\delta\bigr\}.
\end{equation}
The next proposition shows that the metric $d$ on the space of
integrands is comparable to the metric $d_{\mathrm{KF}}$ on the space of integrals.

%
%pr5.4 #&#
%
\begin{proposition}\label{pP-metric}
For all $f\in\LL^+$, we have
%
%e5.10 #&#
%
\begin{equation}
\label{epP-metric-ii} d_{\mathrm{KF}}\bigl(I(f)\bigr) \le2 d(f)^{2/3}\quad\mbox{and}\quad 1-\mathrm{e}^{-c d(f)^2} \le 2 d_{\mathrm{KF}}\bigl(I(f) - I(f)'
\bigr) \le4 d_{\mathrm{KF}}\bigl(I(f)\bigr),
\end{equation}
with $c=1-\sin(1)$, where $I(f)'$ is an independent copy of $I(f)$.
\end{proposition}

The following elementary inequality is used in the proof of Proposition
\ref{pP-metric}.

%le5.5 #&#
%
\begin{lemma}\label{lInequality} Let $X$ be a symmetric random
variable. Then
%
%e5.11 #&#
%
\begin{equation}
\label{elInequality} \sup_{|\theta| \le1} \bigl(1-\E \mathrm{e}^{\mathrm{i}\theta X}\bigr)
\le2 d_{\mathrm{KF}}(X).
\end{equation}
\end{lemma}

\begin{pf*}{Proof of Lemma \ref{lInequality}}
Since $X$ is symmetric, we have that its characteristic function $\phi
_X(\theta) = \E \mathrm{e}^{\mathrm{i}\theta X}, \theta\in\R$ is real and
\[
\bigl(1 - \phi_X(\theta)\bigr) = \int_{-\infty}^\infty
\bigl(1-\cos(\theta x)\bigr) F_X(\mathrm{d}x). %
\]
Note that $0\le1-\cos(u) \le u^2/2$, for all $u\in\R$. Thus, with
$\varepsilon\in(0,1]$, we have
\[
\bigl(1 - \phi_X(\theta)\bigr) \le\frac{|\theta\varepsilon|^2}{2} \int
_{-\varepsilon}^\varepsilon F_X(\mathrm{d}x) + \int
_{|x| \ge\varepsilon} F_X(\mathrm{d}x) \le\varepsilon+ \P\bigl\{ |X| \ge
\varepsilon\bigr\} %
\]
for all $|\theta|\le1\le\sqrt{2/\varepsilon}$. The inequality
(\ref{elInequality}) follows from the definition (\ref{eKyFan}) of
the Ky Fan distance functional.
\end{pf*}

\begin{pf*}{Proof of Proposition \ref{pP-metric}}
We first prove the second inequality in~(\ref{epP-metric-ii}).
Let $X:= I(f) -I(f)'$, where $I(f)'$ is an independent copy of $I(f)$.
Thus, in view of (\ref{ejoint-chf}), $X$ is
symmetric with characteristic function
%
%e5.12 #&#
%
\begin{equation}
\label{eP-metric-phiX} \phi_{X}(\theta) = \bigl|\E \mathrm{e}^{\mathrm{i}\theta I(f)}\bigr|^2
= \exp \biggl\{ -2 \int_\Omega\bigl(1-\cos(\theta f)\bigr) \,\mathrm{d}
\mu \biggr\},\qquad \theta\in\R.
\end{equation}
Now, by Lemma \ref{lInequality}, we obtain
\[
0\le\sup_{|\theta|\le1} \bigl( 1-\phi_{X}(\theta) \bigr)
\le2 d_{\mathrm{KF}}(X). %
\]
Thus, in view of (\ref{eP-metric-phiX}), using the fact that the
function $u\mapsto1-\mathrm{e}^{-2u}, u\ge0$ is strictly
increasing, the above supremum can be taken inside the exponential, and hence
%
%e5.13 #&#
%
\begin{equation}
\label{eP-metric-1} 1 - \mathrm{e}^{-2 A }:= 1- \exp \biggl\{ -2\sup
_{|\theta|\le1} \int_{\Omega} \bigl(1-\cos(\theta f)
\bigr) \,\mathrm{d} \mu \biggr\} \le2 d_{\mathrm{KF}}(X).
\end{equation}
We will focus on the term $A$ above and obtain a lower bound for it.
Notice that
\[
\sup_{|\theta|\le1} \int_{\{|f|\le1\}} \bigl(1-\cos(\theta
f)\bigr) \,\mathrm{d}\mu+ \sup_{|\theta|\le1} \int_{\{|f| >1\}}
\bigl(1-\cos(\theta f)\bigr) \,\mathrm{d}\mu\le A+A\equiv2 A. %
\]
Since $x^2/3 \le1-\cos(x), |x|\le1$, for the first term above, we have
\[
\frac{1}{3}\int_{\{|f| \le1\}} |f|^2 \,\mathrm{d}\mu= \sup
_{|\theta|\le1} \frac
{\theta^2}{3}\int_{\{|f| \le1\}}
|f|^2 \,\mathrm{d}\mu\le\sup_{|\theta|\le1} \int_{\{|f| \le1\}}
\bigl(1-\cos(\theta f)\bigr) \,\mathrm{d} \mu. %
\]
On the other hand, over the set $\{|f| >1\}$, we apply the inequality
$\sup_{|\theta|\le1} (1-\cos(\theta f)) \ge\int_0^1 (1-\cos
(\theta
f)) \,\mathrm{d}\theta= 1-\sin(f)/f$. By combining these two lower bounds,
we obtain
%
%e5.14 #&#
%
\begin{equation}
\label{eP-metric-15} \frac{1}{3}\int_{\{|f| \le1\}} |f|^2
\,\mathrm{d}\mu+ \int_{\{|f| > 1\}}  \biggl(1-\frac{\sin(f)}{f}  \biggr) \,\mathrm{d}
\mu \le2 A.
\end{equation}
Also, since $1-\sin(x)/x \ge1-\sin(1)=:c \approx0.1585 >0$, for all
$|x|\ge1$, we obtain further that
\[
c d(f)^2 \le\frac{1}{3}\int_{\{|f| \le1\}}
|f|^2 \,\mathrm{d}\mu+ \int_{\{|f|
> 1\}}  \biggl(1-
\frac{\sin(f)}{f}  \biggr) \,\mathrm{d} \mu. %
\]
In view of (\ref{eP-metric-1}), (\ref{eP-metric-15}), and the
monotonicity of $u\mapsto1-\mathrm{e}^{-u}$, we obtain
$
1-\mathrm{e}^{- c d(f)^2} \le2 d_{\mathrm{KF}}(X)$,
which, since $d_{\mathrm{KF}}(X) \equiv d_{\mathrm{KF}}(I(f) - I(f)') \le2
d_{\mathrm{KF}}(I(f))$, yields the second inequality in~(\ref{epP-metric-ii}).

We now establish the first inequality in (\ref{epP-metric-ii}). Let
$d:=d(f)\equiv(\int_{\Omega}1\wedge|f|^2 \,\mathrm{d}\mu)^{1/2}, f\in\LL^+$ and consider the sets $A=\{|f|\ge1\}$ and $B=\{|f|<1\}$.
Note that $\mu(A)<\infty$ and recall by (\ref{eqdefidstochint})
that $I(f \ind_A) = \int_{A} f \,\mathrm{d}\Pi_\mu- \int_A a(f) \,\mathrm{d}\mu$. From the
definition of $a$ and $d$, see~(\ref{ea-def}) and~(\ref{ed-def}), it
follows that $|\int_A a(f) \,\mathrm{d}\mu| \leq\mu(A)\leq d^2$
and therefore
%
%e5.15 #&#
%
\begin{equation}
\label{eP-metric-175} \P\bigl\{ \bigl|I(f\ind_A)\bigr| > d^2 \bigr\} \le
\P \biggl\{  \biggl|\int_{A} f \,\mathrm{d}\Pi _\mu  \biggr|\neq0
 \biggr\} \le1-\mathrm{e}^{-\mu(A)} \le1- \mathrm{e}^{-d^2}.
\end{equation}
The second inequality follows from the fact that $\int_{A} f\,\mathrm{d} \Pi_\mu$
is non-zero only when the Poisson point
process $\Pi_{\mu}$ has at least one point in the set $A$.
Also, $I(f \ind_B)$ has (by definition) expectation~$0$ and variance
$\int_B f^2 \,\mathrm{d}\mu\le d^2$. Thus, by the Chebyshev's inequality,
%
%e5.16 #&#
%
\begin{equation}
\label{eP-metric-18} \P\bigl\{ \bigl|I(f\ind_B)\bigr| > d^{2/3} \bigr\}
\leq d^{2/3}.
\end{equation}
Since $I(f) = I(f\ind_A) + I(f\ind_B)$, by (\ref{eP-metric-175}) and
(\ref{eP-metric-18}), in the case $d\leq1$, we get
\begin{eqnarray*}
\P\bigl\{ \bigl|I(f)\bigr| > 2 d^{2/3}\bigr\} &\leq&\P\bigl\{\bigl|I(f)\bigr|
>d^2 + d^{2/3}\bigr\}
\\
&\leq&\P\bigl\{\bigl|I(f\ind_A)\bigr| >d^2\bigr\}+\P\bigl\{\bigl|I(f
\ind_B)\bigr| >{d^{2/3}}\bigr\}
\\
&\leq& 1-\mathrm{e}^{-d^{2}}+d^{2/3}
\\
&\leq& 2 d^{2/3}.
\end{eqnarray*}
Hence $d_{\mathrm{KF}}(I(f)) \le2 d^{2/3}$, provided that $d\leq1$. This,
since $d_{\mathrm{KF}}(I(f)) \le1$ implies the first inequality in
(\ref{epP-metric-ii}).
\end{pf*}

\begin{pf*}{Proof of Proposition~\ref{pLLpluscomplete}}
The proof is standard.
Let $\{f_n\}_{n\in\N}\subset\LL^+$ be a Cauchy sequence in~$d$. Then,
for all $\varepsilon\in(0,1)$, we have
\[
\mu\bigl\{ |f_m - f_n| > \varepsilon\bigr\} \le\frac{1}{\varepsilon^2}
\int_\Omega1\wedge |f_m - f_n|^2
\,\mathrm{d}\mu = \frac{d(f_m,f_n)^2}{\varepsilon^2} \to0, %
\]
as $m,n\to\infty$, which shows that $\{f_n\}_{n\in\N}$ is Cauchy in
measure. Hence, there exists a sub-sequence
$\{n_k\}_{k\in\N}$ and a measurable function $f$, such that $f_{n_k}
\to f$, as $n_k\to\infty$, $\mu$-a.e. Now,
by the Fatou's lemma, we obtain
\[
d(f_{n_k}, f)^2 = \int_{\Omega} 1
\wedge|f_{n_k} - f|^2 \,\mathrm{d}\mu\le \liminf_{\ell\to\infty}
\int_{\Omega} 1\wedge|f_{n_k} - f_{n_\ell}|^2
\,\mathrm{d}\mu = \liminf_{\ell\to\infty} d(f_{n_k},f_{n_\ell})^2.
\]
This inequality implies that $d(f_{n_k},f)<\infty$, and hence $f\in
\LL
^+$, because $d(f,0)\le d(f,f_{n_k}) + d(f_{n_k},0) <\infty$.
Since $\{f_{n}\}_{n\in\N}$ is Cauchy in the metric $d$, we also have
that $d(f_{n_k},f)\to0$, as $n_k\to\infty$, and
hence $d(f_n,f)\to0$, as $n\to\infty$. Thereby proving that the metric
$d$ is complete.

Let now $(\Omega,{\BB})$ be Borel. Recall that the measure $\mu$ is
$\sigma$-finite. Then the space $L^2=L^2(\Omega,{\BB},\mu)$
($\subset
{\LL}^+$)
equipped with the usual $L^2$-norm is separable and let $\{f_n\}_{n\in
\bbN}$ be a dense subset of $L^2$.
By (\ref{ed-def}), for all $f\in{\LL}^+$, $f_n\in L^2$, and $K>0$,
we have
\begin{eqnarray*}
d(f,f_n)^2 &=& \int_{\Omega} 1
\wedge|f-f_n|^2 \,\mathrm{d}\mu
\\
&\leq&\int_{\{|f|\le K\}} |f-f_n|^2 \,\mathrm{d}\mu+ \int
_{\{ |f|>K\}} \mathrm{d}\mu
\\
&\leq&\int_{\Omega} ( f\ind_{\{|f|\le K\}} -
f_n)^2 \,\mathrm{d}\mu+ \mu\bigl\{|f| >K\bigr\}.
\end{eqnarray*}
Since $f\in{\LL}^+$, we have that $f\ind_{\{|f|\le K \}} \in L^2$ and
$\mu\{|f|>K\} \to0$, as $K\to\infty$. Thus,
by picking large enough $K$ and a suitable $f_n$, one can make
$d(f,f_n)$ arbitrarily small, showing
that $\{f_n\}_{n\in\bbN}$ is also dense in the metric space $({\LL
}^+,d)$, thereby proving separability.
\end{pf*}

\begin{pf*}{Proof of Proposition \ref{pcauchyprobabcauchyLplus}}
Suppose first that $d(f_n-f) + |c_n-c| \to0$, as $n\to\infty$.
Then, by Slutsky's theorem, it is enough to show that $I(f_n)$
converges in probability to $I(f)$, as $n\to\infty$.
By (\ref{egamma-def}), we have that $I(f_n) - I(f) = I (f_n - f) +
\gamma(f_n,-f)$. Proposition \ref{pP-metric} and the assumption
$d(f_n-f)\to0$
imply that $I(f_n-f) \stackrel{\P}{\to} 0, n\to\infty$. It
remains to
show that $\gamma(f,-f_n) \to0$, as
$n\to\infty$. By the triangle inequality for $d$, we have $|d(f_n)
-d(f)| \le d(f_n - f) \to0$, as $n\to\infty$, and
in particular $d(f_n), n\in\N$ is bounded. Thus, by Lemma \ref{lgamma}
applied to $f$ and $g:=-f_n$,
we obtain $\gamma(f,-f_n) \to0$, as $n\to\infty$. This completes proof
of the `if' part.

To prove the `only if' part, suppose that $I(f_n)+c_n\stackrel{\P
}{\to}
\xi, n\to\infty$, set
$\xi_{m,n}:= I(f_m) - I(f_n) + c_m - c_n$, and let $\xi_{m,n}'$ be and
independent copy of $\xi_{m,n}$.
Then, by using (\ref{ejoint-chf}) we obtain that
\[
\xi_{m,n} - \xi_{m,n}' \stackrel{d}
{=}I(f_m-f_n) - I(f_m-f_n)',
\]
where $I(f_m-f_n)'$ is an independent copy of $I(f_m-f_n)$. Now, by the
second bound in
(\ref{epP-metric-ii}) of Proposition \ref{pP-metric} applied to
$f:=f_m-f_n$, we obtain
\[
1- \mathrm{e}^{-c d(f_m-f_n)^2} \le2 d_{\mathrm{KF}} \bigl(I(f_m-f_n)
- I(f_m-f_n)'\bigr) \equiv
2d_{\mathrm{KF}} \bigl(\xi_{m,n} - \xi_{m,n}'
\bigr) \le4 d_{\mathrm{KF}}(\xi_{m,n}). %
\]
The right-hand side of the last inequality vanishes, as $m, n\to\infty
$, since the sequence $\{I(f_n) +c_n, n\in\N\}$ converges in probability
and therefore it is Cauchy in the Ky Fan metric. This implies that
$d(f_m -f_n)\to0, m,n\to\infty$, and since $(\LL^+,d)$
is complete (Proposition \ref{pLLpluscomplete}), there is an $f\in\LL
^+$, such that $d(f_n-f)\to0, n\to\infty$. Therefore, by the already
established `if' part, it follows that $I(f_n)\stackrel{\P}{\to}
I(f), n\to\infty$. This, and the fact that $I(f_n) + c_n \stackrel{\P
}{\to}
\xi, n\to\infty$ imply (by Slutsky) that the sequence $c_n$ converges
to a constant $c$ and $\xi= I(f) +c$.
This completes the proof.
\end{pf*}

\begin{pf*}{Proof of Theorem~\ref{theominspecrepsumexists}}
Let $T_0$ be the at most countable subset of $T$ appearing in Condition~\textup{S}. Consider the space
$\bbR^{T_0}$, equipped with the product $\sigma$-algebra $\BB$.
Following \cite{maruyama70} (see also \cite{roy07}), let $\mu$ be the
L\'evy
measure of $\{X(t), t\in T_0\}$ on $\bbR^{T_0}$.
For $t\in T_0$, we define the coordinate mappings $f_t\dvtx \R^{T_0}\to\R$
by $f_t(\varphi) = \varphi(t)$,
where $\varphi\dvtx  T_0\to\R$, $\varphi\in\R^{T_0}$. Then, $\{f_t,
t\in
T_0\}$ is a spectral representation of $\{X(t), t\in T_0\}$ by the
properties of the L\'evy measure.

For $t\notin T$, observe that by Condition~\textup{S}, there exists a
sequence $\{t_n\}\subset T_0$, such that $X(t_n)$ converges in
probability to $X(t)$, as $n\to\infty$. In other words, $I(f_n)+c_n$
converges in probability to $I(f)+c$, for some $c_n$ and $c$. Thus, by
Proposition
\ref{pcauchyprobabcauchyLplus}, the sequence of functions $f_{t_n}$ has
a limit in $({\LL}^+,d)$, as
$n\to\infty$. We take this limit to be the spectral function $f_t$.

Notice that the so-defined spectral representation is minimal. Indeed,
the $\sigma$-algebra $\sigma\{f_t, t\in T\}$
coincides with the product $\sigma$-algebra $\BB$ on $\R^{T_0}$. We
also have that
$\operatorname{supp}\{f_t, t\in T_0\} = \R^{T_0}$ ($\mod \mu$) because $\bigcap_{t\in T_0}\{f_t=0\}=\{0\}$, a set
whose L\'evy measure is $0$ by convention. To \mbox{complete} the proof,
observe that the measurable space
$(\bbR^{T_0},\BB)$ is Borel by Kuratowski's theorem.
\end{pf*}

\begin{pf*}{Proof of Theorem~\ref{theominspecrepsumunique}}
We are going to apply Lemma~\ref{lft-point}. Define the measurable
mappings $F_i\dvtx  (\Omega_i, \BB_i)\to(\R^T, \BB)$ by
\[
F_i(\omega)=\bigl(f^{(i)}_t(\omega)
\bigr)_{t\in T},\qquad \omega\in\Omega_i, i=1,2. %
\]
Minimality implies that the first condition of Lemma~\ref{lft-point} is
satisfied. We prove that $\mu_1\circ F_1^{-1}=\mu_2\circ F_2^{-1}$.
Let $t_1,\ldots,t_n \in T$ and observe that
in view of (\ref{ejoint-chf}) we have
\begin{eqnarray*}
\E \mathrm{e}^{\mathrm{i}\sum_{j=1}^n \theta_j X(t_j)} &=& \E\exp \Biggl\{ \mathrm{i}\sum_{j=1}^n
\theta_j \bigl(I\bigl(f_{t_j}^{(i)}
\bigr)+c_j^{(i)}\bigr) \Biggr\}
\\
&=& \exp \Biggl\{ \mathrm{i}\sum_{j=1}^n
c_j^{(i)} \theta_j + \int_{\bbR^n}
 \Biggl( \mathrm{e}^{\mathrm{i}\sum_{j=1}^n \theta_j x_j} - \mathrm{i} \sum_{j=1}^n
\theta_j a(x_j) - 1  \Biggr) \bigl(\mu_i
\circ G_{i}^{-1}\bigr) (\mathrm{d}x)  \Biggr\},
\end{eqnarray*}
where $G_i = (f_{t_j}^{(i)})_{j=1}^n\dvtx \Omega_i \to\bbR^n$ and
$c_1^{(i)}, \ldots, c_{n}^{(i)}\in\R$ are constants, $i=1,2$. The last
relation and the uniqueness of the L\'evy measure of the i.d. random
vector $(X(t_j))_{j=1}^n$
shows that $(\mu_1 \circ G_1)^{-1}(A) = (\mu_2\circ G_2)^{-1}(A)$ for
all Borel sets $A\subset\R^n\setminus\{0\}$. We need to show that
$(\mu_1 \circ G_1)^{-1}(\{0\}) = (\mu_2\circ G_2)^{-1}(\{0\})$. As in
the proof of Theorem~\ref{theominspecrepmaxunique} we can find a
sequence $q_1,q_2,\ldots \in T$ such that $\mu_i(\bigcap_{j\in\N}\{
f_{t_j}^{(i)} \})=0$, $i=1,2$. Consider measurable sets
\[
E_{i,p}= G_i^{-1}\bigl(\{0\}\bigr) \cap\Biggl(
\bigcap_{j=1}^{p-1} \bigl\{f_{q_j}^{(i)}=0
\bigr\}\Biggr) \cap\bigl\{f_{q_p}^{(i)}\neq0\bigr\}.
\]
For every $p$, we have shown that $\mu_1(E_{1,p})=\mu_2(E_{2,p})$. It
follows that
\[
\mu_1\bigl(G_1^{-1}\bigl(\{0\}\bigr)\bigr) =
\sum_{p=1}^{\infty} \mu_1(E_{1,p})=
\sum_{p=1}^{\infty} \mu_2(E_{2,p})=
\mu_2\bigl(G_2^{-1}\bigl(\{0\}\bigr)\bigr).
\]
This\vspace*{1pt} proves that $(\mu_1 \circ G_1)^{-1}(A) = (\mu_2\circ
G_2)^{-1}(A)$ for all Borel sets $A\subset\R^n$. In other words, the
measures $\mu_1\circ F_1^{-1}$ and $\mu_2\circ F_2^{-1}$ are equal on
the semiring $\mathcal C$ consisting of subsets $\{\varphi\dvtx T\to\R\dvtx
(\varphi(t_j))_{j=1}^n \in A\}$, where $A\subset\R^n$ is Borel. This
semiring generates the product $\sigma$-algebra $\BB$. Also, we have a
decomposition
\[
\R^T=\bigcup_{n=1}^{\infty}\bigcup
_{k=1}^{\infty} \bigl\{\varphi\dvtx T\to\R\dvtx
k^{-1}\leq\bigl|\varphi(q_n)\bigr|\leq k \bigr\} \mod
\mu_1\circ F_1^{-1} \mbox{ and }
\mu_2\circ F_2^{-1}. %
\]
Note that the sets on the right-hand side have finite $\mu_1\circ
F_1^{-1}$ (and $\mu_2\circ F_2^{-1}$) measure and belong to the
semiring $\mathcal C$.
By the uniqueness of measure extension theorem, the measures $\mu
_1\circ F_1^{-1}$ and $\mu_2\circ F_2^{-1}$ are equal. Lemma~\ref
{lft-point} completes the proof.
\end{pf*}

%s5.4 #&#
\subsection{Proof of Theorem~\texorpdfstring{\protect\ref{teoC-int-test}}{3.9}}
By Theorems \ref{theominspecrepmaxexists}, \ref
{theominspecrepsumexists} and \ref{theoflowrepexists}, the process $X$
has a minimal spectral
representation $g_t:= g_0 \circ T_t, t\in\T^d$ over a $\sigma$-finite
Borel space $(\widetilde\Omega,\widetilde \BB,\widetilde \mu)$,
where $\{
T_t, t\in\T^d\}$ is a measure
preserving and measurable flow (see also Proposition \ref{pmeasurability}).

Since the spectral representation $\{f_t, t\in T\}\subset\LL^{\vee
/+}(\Omega,\BB,\mu)$ is of full support, it is minimal if we set
${\mathcal
B} = \sigma\{ f_t, t\in T\}$.
Even though $\BB$ may not be Borel, Theorems \ref
{theominspecrepmaxunique}\textup{(i)} and \ref{theominspecrepsumunique}\textup{(i)}
imply that there exists a measurable measure-preserving mapping $\Phi\dvtx
(\Omega,\BB)\to(\widetilde\Omega,\widetilde \BB)$, such that for all
$t\in\T^d$,
$g_t\circ\Phi= f_t$ $\mu$-a.e. Therefore, by using the joint
measurability of the two representations and appealing to Fubini, we
see that
%
%e5.17 #&#
%
\begin{eqnarray}
\label{ef-via-g}
\int_{\T^d} \psi\bigl(\bigl|f_t(
\omega)\bigr|\bigr) \lambda(\mathrm{d}t) &=& \int_{\T^d} \psi\bigl( \bigl|
g_t\bigl( \Phi(\omega)\bigr) \bigr|\bigr) \lambda(\mathrm{d}t)
\nonumber\\[-8pt]\\[-8pt]
&\equiv& \int _{\T^d} \psi\bigl( \bigl| g_0\circ T_t
\bigl(\Phi(\omega)\bigr)\bigr|\bigr) \lambda(\mathrm{d}t) = \infty,\qquad \mu\mbox{-a.e.}\nonumber
\end{eqnarray}
which,\vspace*{1pt} since $\widetilde \mu= \mu\circ\Phi^{-1}$, shows that relation
(\ref{eC-test}) is equivalent to $\int_{\T^d} \psi(|g_0 \circ T_t
(\widetilde \omega)|)\lambda(\mathrm{d}t) = \infty$,
$\widetilde \mu$-a.e. Thus, using the criterion in Theorem \ref
{teoRoy-C-test} one can relate (\ref{eC-test}) to the conservativity of
the flow. More precisely, proceeding as in the
proof of Proposition 3.2 in \cite{roy2010}, let
\[
h(\widetilde \omega):= \sum_{\gamma\in\Z^d}
a_\gamma\int_{\gamma+
[0,1)^d} \psi\bigl(\bigl|g_0\circ
T_t (\widetilde \omega)\bigr|\bigr) \lambda(\mathrm{d}t), %
\]
where\vspace*{1pt} $a_\gamma>0$ and $\sum_{\gamma\in\Z^d} a_\gamma=1$. By Fubini's
theorem, the full support condition on $\{g_t, t\in T\}$
implies that $h\in\LL^1(\widetilde \Omega,\widetilde \BB,\widetilde \mu)$ and
$h>0$, $\widetilde \mu$-a.e. Observe also by applying Fubini again and
using the facts that
$\lambda$ is shift-invariant and the flow $\{T_t\}_{t\in\T^d}$ is
measure-preserving
\[
\sum_{\beta\in\Z^d} h\circ T_\beta(\widetilde
\omega) = \int_{\T^d} \psi \bigl(\bigl|g_0\circ
T_t(\widetilde \omega)\bigr|\bigr) \lambda(\mathrm{d}t). %
\]
Theorem \ref{teoRoy-C-test}, applied to the discrete flow $\{T_\beta\}
_{\beta\in\Z^d}$ shows that is conservative if and only if (\ref
{eC-test}) holds, which
completes the proof.

\section*{Acknowledgements}

We thank the editor-in-chief and an anonymous associate editor for
handling our paper. We are also grateful for exceptionally detailed
reports of two anonymous referees,
which helped us significantly improve the content and presentation of
the results. Silian A.~Stoev was partially supported by the NSF Grant
DMS-1106695 at the University of Michigan.

%\begin{appendix}
%\section{}
%\end{appendix}

%\section*{Acknowledgements}

%\begin{supplement}%[id=suppA]
%\slink[doi]{10.3150/00-BEJXXXXSUPP} %[doi,text={...}] - jei reikia
%suskaldyti doi
%\sfilename{BEJ000\_supp.pdf}
%\end{supplement}

% imsref loaded by linak, 2014-06-10 11:09:19
%

\printhistory
\end{document}